\documentclass[11pt]{article}

\usepackage{amsmath,amssymb,xfrac}
\usepackage{algorithm,algorithmic}
\usepackage[margin=1in]{geometry}
\usepackage{url}
\usepackage{graphicx}
\usepackage{tikz}
\usetikzlibrary{arrows,positioning,shapes.misc}
\usepackage[square,numbers]{natbib}
\usepackage{amsthm}
\newtheorem{theorem}{Theorem}
\newtheorem{lemma}{Lemma}
\newtheorem{proposition}{Proposition}
\newtheorem{corollary}{Corollary}
\newtheorem{assumption}{Assumption}
\theoremstyle{definition}

\newcommand{\mbf}{\mathbf}
\newcommand{\mbb}{\mathbb}

\newcommand{\st}{\mathrm{s.t.}\;}
\newcommand{\tr}{^{\mathrm{T}}}
\newcommand{\inv}{^{-1}}
\newcommand{\abs}[1]{\left| #1 \right|}
\newcommand{\norm}[1]{\left\| #1 \right\|}
\newcommand{\set}[1]{\left\{ #1 \right\}}
\newcommand{\seq}[2][]{\left( #2 \right)_{#1}} 
\newcommand{\grad}{\nabla}
\newcommand{\smallsum}{{\textstyle{\sum}}}
\renewcommand{\hat}{\widehat}
\renewcommand{\tilde}{\widetilde}

\newcommand{\symgt}{\succ}

\newcommand{\0}{\mathbf{0}}

\renewcommand{\b}{\mathbf{b}}
\renewcommand{\c}{\mathbf{c}}

\newcommand{\p}{\mathbf{p}}
\newcommand{\q}{\mathbf{q}}
\renewcommand{\r}{\mathbf{r}}
\newcommand{\s}{\mathbf{s}}
\renewcommand{\u}{\mathbf{u}}

\newcommand{\x}{\mathbf{x}}
\newcommand{\y}{\mathbf{y}}
\newcommand{\z}{\mathbf{z}}

\newcommand{\blam}{\boldsymbol{\lambda}}
\newcommand{\bmu}{\boldsymbol{\mu}}

\newcommand{\diag}{\mathbf{Diag}}

\title{Analysis of the alternating direction method of multipliers for nonconvex problems%
\thanks{This is a slightly edited version of the article that was accepted for publication in Operations Research Forum. 
The Version of Record is available online at: \protect\url{http://dx.doi.org/10.1007/s43069-020-00043-y}.}
}
\author{Stuart M. Harwood%
\thanks{ExxonMobil Research and Engineering\newline
stuart.m.harwood@exxonmobil.com\newline
ORCID: 0000-0001-5883-9624}
}
\date{\today}

\begin{document}
\maketitle

\begin{abstract}
This work investigates the theoretical performance of the alternating-direction method of multipliers (ADMM) as it applies to nonconvex optimization problems, and in particular, problems with nonconvex constraint sets.
The alternating direction method of multipliers is an optimization method that has largely been analyzed for convex problems.
The ultimate goal is to assess what kind of theoretical convergence properties the method has in the nonconvex case, and to this end, theoretical contributions are two-fold.
First, this work analyzes the method with \emph{local} optimal solution of the ADMM subproblems, which contrasts with much analysis that requires global solutions of the subproblems.
Such a consideration is important to practical implementations.
Second, it is established that the method still satisfies a local convergence result.
The work concludes with some more detailed discussion of how the analysis relates to previous work.
\end{abstract}


\section{Introduction}
\label{sec:intro}

This work analyzes the alternating direction method of multipliers (ADMM) as it applies to nonconvex problems.
Motivation to consider ADMM is its applicability to distributed or nearly-separable systems
(see problem forms~\eqref{problem_blocks} and \eqref{problem_blocks_noy} in \S\ref{sec:prob_statement}).
Related work has focused on solving systems of equations in a parallelizable way.
For instance, \cite{zhangEA92} considers how to solve systems of equations whose Jacobian has an ``arrowhead'' type sparsity structure.
Similar considerations in the linear algebra routine of an interior-point-type method for nonlinear programming appear in \cite{chiangEA14,kangEA14}.
Another approach specific to quadratic optimization problems is considered in \cite{curtisEA17}, taking advantage of a nonsmooth reformulation.

However, it can be theoretically informative and practically advantageous to consider methods that decompose the optimization problem, rather than decompose underlying linear algebra.
This leads to methods of a primal-dual nature, relying on augmented Lagrangians and ``local'' strong duality, namely, the method of multipliers.
Early work goes back to \cite{stephanopoulosEA75}.
This work uses a quadratic penalty term, which allows strong duality to hold when the problem is restricted to a neighborhood of a local minimum.
The penalty term does not preserve separable structure, which motivated an approximation based on a Taylor expansion.
This motivates the development of ``separable'' augmented Lagrangians, the spirit of which is introduced in \cite{bertsekas79}.
In that work, new variables are introduced and a quadratic penalty is added to the objective, essentially penalizing the distance between the original variables and the new ones. 
This preserves separability of the resulting (augmented) Lagrangian, with the new variables treated through an extra level of optimization.
As noted in \cite{tanikawaEA85,fengEA90}, this results in three levels of optimization to take advantage of decomposable structure through a primal-dual method like the method of multipliers.
Consequently, the methods proposed in \cite{tanikawaEA85,fengEA90} take advantage of Fletcher's multiplier estimate to eliminate one level of optimization.
More recent work includes \cite{dinhEA13,hoursEA14}.
In \cite{hoursEA14}, they propose using a regularized block-coordinate descent method to solve the minimization of the augmented Lagrangian in the method of multipliers.
In \cite{dinhEA13}, a sequential convex programming approach is taken, and distributable methods for convex programming are used on the subproblems.

Recently, there has been a surge of work on the alternating direction method of multipliers.
This method has been well characterized for convex problems (see \cite{boydEA11} for a review).
As noted in \cite{ecksteinEA_draft}, one view of ADMM is that it is like applying a single iteration of a block-coordinate descent method to the minimization of the augmented Lagrangian in the method of multipliers.
Consequently ADMM naturally accommodates problems with a decomposable structure.
However, computational experiments in \cite{ecksteinEA_draft} suggest that ADMM is overall more computationally efficient than a ``true'' approximate method of multipliers.

Recent work has focused on applying ADMM to nonconvex problems, and establishing that it still converges.
This recent work  includes \cite{baiEA_draft}, which focuses on problems with more specific structure, such as objectives which are convex in one block of variables with the others fixed, or quadratic (but not necessarily convex).
That work also does not consider nonconvex constraints.
Meanwhile, the work in \cite{liEA15} does not explicitly consider separable constraints, but it allows for some nonsmoothness in the objective, which could allow us to handle the separable constraints through an exact nonsmooth penalty function.
Further, \cite{wangEA15} establishes convergence of ADMM under very general conditions, including the setting when there are nonconvex constraints.
However, the resulting method in this situation relies on global solution of the subproblems and the constraint set having a tractable projection operation.

The recent work in \cite{chatzipanagiotisEA15,chatzipanagiotisEA17,hongEA16,houskaEA16,magnussonEA15} all deal with ADMM-type methods in various settings.
The analysis in \cite{hongEA16} allows for nonconvex objectives, but in contrast with the present work, assumes convex constraint sets and that global solutions of the subproblems are found.
The methods considered in \cite{magnussonEA15} do not require global solution of the subproblems, like the present work, but the convergence results for ADMM in the constrained, nonconvex case, are more a statement of ``correctness'' of the method;
that is, if the iterates converge, then they converge to a stationary point.
The present work will establish conditions under which convergence occurs.
Meanwhile, \cite{chatzipanagiotisEA15,chatzipanagiotisEA17} focus on local convergence, like the present work, but while nonconvex objectives are allowed, convex constraint sets are still assumed.
Further, the methods proposed in \cite{chatzipanagiotisEA15,chatzipanagiotisEA17} require the use of a stepsize, the value of which is critical to the methods' convergence (see also \S\ref{sec:discussion} for further discussion).
Finally, the recent work in \cite{houskaEA16} considers problem form~\eqref{problem_blocks_noy}, and allows for nonconvex objective and constraints.
In contrast with the present work, the method proposed in \cite{houskaEA16} is a nontrivial extension of ADMM.
These modifications (improved derivative information and a linesearch) are interesting and provide insight into ways that the robustness and convergence of ADMM might be improved.
However, the goal of this work is different;
the aim is to prove convergence properties of a standard form of ADMM with no modification.

\textbf{Contribution and structure.}
The goals and contributions of this work are to analyze a standard form of ADMM and establish that it converges even when applied to a fairly general, nonconvex problem.
Furthermore, the assumptions in this analysis aim to be as permissive as possible when it comes to the solution of the ADMM subproblem at each iteration
(see problem~\eqref{ADMMSP} in the statement of the method in the following section).
Specifically, the method does not specify how the subproblem must be solved, or how the nonconvex constraints 
must be handled.
Since local optimal solutions are permitted, this allows for very powerful numerical methods for the solution of general nonlinear optimization problems to be used.
For instance, \cite{rodriguezEA18} performs numerical studies of ADMM and other methods in the nonconvex setting.
Their numerical implementation uses the powerful optimization solver Ipopt%
\footnote{The Ipopt project page and source code is available at \url{https://github.com/coin-or/Ipopt}.}
\cite{wachterEA06} and achieve promising results for the performance of ADMM in certain settings.
They also point out that progressive hedging in the stochastic programming literature is an application of ADMM.
Thus the present work provides a theoretical basis for the convergence of progressive hedging in the nonconvex case.

The rest of this work is organized as follows.
The method is stated in \S\ref{sec:statement}.
A preliminary analysis in \S\ref{sec:prelim} defines relevant quantities and establishes some properties that are useful in \S\ref{sec:localConv}, which contains the main local convergence result, Theorem~\ref{thm:convergence}.
Section~\ref{sec:discussion} discusses the analysis and how it connects with other work, and finally a few examples are worked out in \S\ref{sec:examples}.

\section{Statement and preliminaries}
\label{sec:statement}

\subsection{Problem statement, notation}
\label{sec:prob_statement}
The problem of interest is the optimization problem
\begin{align}
\label{problem}
\tag{P}
\min_{\x,\y}\; & f(\x) \\
\st
\notag & \c(\x) = \0, \\
\notag & \mbf{A}\x + \mbf{B}\y = \b
\end{align}
where
$f : \mbb{R}^{n} \to \mbb{R}$,
$\mbf{c} : \mbb{R}^{n} \to \mbb{R}^{p}$,
$\mbf{A} \in \mbb{R}^{q \times n}$, 
$\mbf{B} \in \mbb{R}^{q \times m}$, and
$\b \in \mbb{R}^q$.
No assumption is made that $f$ is convex or $\c$ affine.
Instead, we will require problem regularity in the form of a local minimizer satisfying the second order sufficient conditions.

Notation includes $\0$, which may denote a vector of zeros or a matrix of zeros;
whether it is a vector or matrix should be clear from context.
Otherwise, bold uppercase letters denote matrices (or matrix-valued mappings), while bold lowercase letters denote vectors (or vector-valued mappings).
The Jacobian matrix of the vector-valued mapping $\c$ evaluated at $\x$ is denoted $\grad \c(\x) \tr$
(where superscript $\mathrm{T}$ denotes the transpose).
Further, $\norm{\cdot}$ denotes the standard (Euclidean) 2-norm;
different norms will be defined as needed and distinguished with a subscript.
$N_{\epsilon} (\z) = \set{ \z' : \norm{\z' - \z} < \epsilon}$
denotes an open ball.
A sequence $(x^0, x^1, x^2 ,\dots)$ may be abbreviated $\seq[k \in \mbb{N}]{x^k}$ or just $\seq[k]{x^k}$.
We treat the case of only equality constraints to avoid extra notational burden;
as far as the theory is concerned, any inequalities $h(\x) \le 0$ can be transformed by adding free variables $s$ and writing the constraint as 
$h(\x) + s^2 = 0$
(see for instance the discussion in \cite[\S3.3.2]{bertsekas_nlp}).

Compared with other problem forms typically considered in the study of ADMM, Problem~\eqref{problem} is missing an extra function of $\y$ in its objective.
This omission permits an explicit algebraic form for the $\y$ iterate update in ADMM and simplifies some analysis.
Further, we can ``absorb'' an objective function of $\y$ into the function of $\x$ at the expense of extra variables and constraints:
\[
\begin{aligned}
\min_{\x,\y}\; & f(\x) + g(\y)\\
\st
& \mbf{A}\x + \mbf{B}\y = \b
\end{aligned}
\iff
\begin{aligned}
\min_{\x,\x^y,\y}\; & (f(\x) + g(\x^y))\\
\st
& \mbf{A}\x + \mbf{B}\y = \b, \\
& \x^y = \y.
\end{aligned}
\]
Problem~\eqref{problem} affords a lot of flexibility, and motivation for this form is apparent when we consider that $\x$ may be naturally partitioned into nearly independent ``blocks'' of variables;
for instance, the problem form considered in \cite{rodriguezEA18} is
\begin{align}
\label{problem_blocks}
\min_{\x_1, \x_2, \dots, \x_N,\y} & \sum_{i=1}^N f_i(\x_i) \\
\st
\notag & \c_i(\x_i) = \0, \forall i, \\
\notag & \mbf{A}_i \x_i + \mbf{B}_i \y = \b_i, \forall i.
\end{align}
This problem fits into the form of Problem~\eqref{problem} by setting
$\x = (\x_1, \dots, \x_N)$,
$f : \x \mapsto \sum_i f_i(\x_i)$,
$\c : \x \mapsto (\c_1(\x_1), \dots, \c_N(\x_N))$, 
$\mbf{A} = \diag_i(\mbf{A}_i)$, 
$\mbf{B} = \begin{bmatrix} \mbf{B}_1\tr & \dots & \mbf{B}_N\tr \end{bmatrix}\tr$, and
$\b = (\b_1, \dots, \b_N)$.
This form motivates the inclusion of the variables $\y$, which may be viewed as complicating variables.
For instance, in the setting of two-stage stochastic programming, $\y$ takes the role of first-stage decisions.

Another problem form, considered in \cite{chatzipanagiotisEA17,houskaEA16} is
\begin{align}
\label{problem_blocks_noy}
\min_{\x_1, \x_2, \dots, \x_N} & \smallsum_i f_i(\x_i) \\
\st
\notag & \c_i(\x_i) = \0, \forall i, \\
\notag & \smallsum_i \mbf{A}_i \x_i = \tilde\b.
\end{align}
This problem can also be put into the form of Problem~\eqref{problem} by setting
$\x = (\x_1, \dots, \x_N)$,
$f : \x \mapsto \sum_i f_i(\x_i)$,
$\c : \x \mapsto (\c_1(\x_1), \dots, \c_N(\x_N))$, 
\[
\mbf{A} = 
\begin{bmatrix}
\mbf{A}_1 	& 			& \\
          	& \ddots 	& \\
			& 			& \mbf{A}_N \\
\0			& \dots  	& \0
\end{bmatrix},
\qquad
\mbf{B} = 
\begin{bmatrix}
-\mbf{I} 	& 			& \\
          	& \ddots 	& \\
			& 			& -\mbf{I} \\
\mbf{I}		& \dots  	& \mbf{I}
\end{bmatrix},
\text{ and}
\qquad
\b = 
\begin{bmatrix}
\0 \\
\vdots \\
\0 \\
\tilde\b
\end{bmatrix}.
\]

\subsection{Method statement}
Algorithm~\ref{alg:ADMM} states ADMM as it would apply to Problem~\eqref{problem}.
This is the method we will analyze.
Its form is consistent with ADMM considered in many previous studies, including \cite{boydEA11,hongEA16,magnussonEA15,rodriguezEA18}.
Note that solution of the subproblem~\eqref{ADMMSP}, the most computationally intensive step, may be decomposed (and parallelized) in the situation that the problem is derived from Problem~\eqref{problem_blocks} or \eqref{problem_blocks_noy}.
The quantities $\q^k$ and $\r^k$ defined in the method are called the primal and dual residuals, respectively.
Their interpretation as such will be justified shortly.
Inputs to the algorithm include the tolerances $\eta^p$, $\eta^d$, and the algorithm terminates once the norms of the residuals are less than these tolerances.
Other inputs include initial guesses $\y^0$ and $\blam^0$ for the $\y$ variables and multipliers of the linear constraints, respectively, as well as the penalty parameter $\rho$ of the augmented Lagrangian.

\begin{algorithm}[t]
\caption{ADMM for Problem~\eqref{problem}}
\label{alg:ADMM}
\begin{algorithmic}
\REQUIRE $\y^0 \in\mbb{R}^m$, $\blam^0\in\mbb{R}^q$, 
			$\rho > 0$, $\eta^p > 0$, $\eta^d > 0$, 
			such that $\mbf{B}\tr \blam^0 = \0$
\FOR{$k \in \set{0,1,2,\dots}$}
	\STATE
	Find a local minimum $\x^{k+1}$ of
	\begin{align}
	\label{ADMMSP}
	\tag{SP}
	\min_{\x} \;
	& f(\x) + 
		(\blam^k) \tr      (\mbf{A}\x + \mbf{B}\y^k - \b) + 
		\frac{\rho}{2}\norm{\mbf{A}\x + \mbf{B}\y^k - \b}^2 \\
	\st
	& \c(\x) = \0 \notag,
	\end{align}
	\STATE
	Let:
	\STATE
	$\y^{k+1} \in 
	\arg \min_{\y} \set{
	(\blam^k)\tr       (\mbf{A}\x^{k+1} +\mbf{B}\y - \b) + 
	\frac{\rho}{2}\norm{\mbf{A}\x^{k+1} +\mbf{B}\y - \b}^2 }
	$
	\STATE
	$\blam^{k+1} = \blam^k + \rho(\mbf{A}\x^{k+1} + \mbf{B}\y^{k+1} - \b)$,
	\STATE
	$\q^{k+1} = \mbf{A}\x^{k+1} + \mbf{B}\y^{k+1} - \b$,
	\STATE
	$\r^{k+1} = \rho\mbf{A}\tr\mbf{B}(\y^{k+1} - \y^k)$,
	\IF{$\norm{\q^{k+1}} \le \eta^p$ \AND $\norm{\r^{k+1}} \le \eta^d$}
	\STATE Terminate.
	\ENDIF
\ENDFOR
\end{algorithmic}
\end{algorithm}

\subsection{Assumptions}
\label{sec:assm}
Here we collect all the assumptions used in this work for easy reference.
However, they will be referenced as necessary in the statement of results to make it clear when they are used.

The first is a basic regularity assumption of the minimizer of the subproblem~\eqref{ADMMSP} found at each iteration.
%
\begin{assumption}
\label{assm:KKTiterates}
Assume that
$f$ and $\c$ are continuously differentiable, and
for all iterations $k$, a KKT point $(\x^{k+1},\bmu^{k+1})$ of subproblem~\eqref{ADMMSP} is found:
\begin{subequations}
\label{spKKT}
\begin{align}
\label{spKKT1} 
&\grad f(\x^{k+1}) + 
	\grad \c(\x^{k+1}) \bmu^{k+1} + 
	\mbf{A} \tr (\blam^k + \rho(\mbf{A}\x^{k+1} +\mbf{B}\y^k - \b)) = \0, \\
\label{spKKT2} 
&\c(\x^{k+1}) = \0.
\end{align}
\end{subequations}
\end{assumption}

The next assumption is a statement that the overall problem has a local minimizer $(\x^*,\y^*)$ which satisfies a constraint qualification and the second-order sufficiency conditions.
Although not stated explicitly, under the conditions in this assumption, the point $(\x^*,\y^*)$ must be a local minimum by, for instance, \cite[Proposition~3.2.1]{bertsekas_nlp}.
In the following we use the notation $c_{j}$, meaning the $j^{th}$ component of $\c$, and similarly for $\mu_{j}$.
\begin{assumption}
\label{assm:OverallRegularity}
Assume that $f$ and $\c$ are twice continuously differentiable and that $(\x^*, \y^*, \bmu^*, \blam^*)$ is a KKT point of Problem~\eqref{problem}:
\begin{subequations}
\label{mainKKT}
\begin{align}
\label{mainKKTx} 
&\grad f(\x^*) + \grad \c(\x^*) \bmu^* + \mbf{A} \tr \blam^* = \0, \\
\label{mainKKTy}
&\mbf{B} \tr \blam^* = \0, \\
\label{mainKKTfeas1} 
&\c(\x^*) = \0, \\
\label{mainKKTfeas2} 
&\mbf{A}\x^* + \mbf{B}\y^* = \b.
\end{align}
\end{subequations}
Assume that the linear independence constraint qualification holds at $(\x^*,\y^*)$: 
the rows of the matrix 
\[
\mbf{C} = 
\begin{bmatrix}
\grad \c(\x^*) \tr & \0 \\
\mbf{A}				& \mbf{B}
\end{bmatrix}
\]
are linearly independent.
Furthermore, let
$\mbf{H}_{xx} : (\x,\bmu) \mapsto \grad^2 f(\x)  + \smallsum_j \mu_j \grad^2 c_j(\x)$
and assume that
\[
\z \tr \begin{bmatrix} \mbf{H}_{xx}(\x^*,\bmu^*) & \0 \\ \0 & \0 \end{bmatrix} \z > 0
\]
for all $\z \neq \0$ satisfying 
$\mbf{C}\z = \0$.
\end{assumption}

Note that the positive definiteness of the Hessian on the null space of the constraint Jacobian implies that $\mbf{B}$ needs to have full column rank%
\footnote{Let $\z_y \neq \0$ and consider $\z = (\0,\z_y)$.
The only way that the positive definiteness condition can hold is if $\mbf{B}\z_y \neq \0$.
This implies that $\mbf{B}$ must have full column rank.
}.

Since we are allowing local minimization of the subproblem \eqref{ADMMSP} in the algorithm, the following assumption aims to resolve which local minimizer of the subproblem is found, without assuming explicitly that it is found in a specific neighborhood.
This is done by assuming that the minimizer of the subproblem that is \emph{closest} to the desired solution $\x^*$ is found at each iteration.
This is similar to assumptions made in the local convergence analysis of the classic method of multipliers (see \cite[\S2.2.4]{bertsekas_lmm}).
Further, this assumption holds if the \emph{subproblem} \eqref{ADMMSP} has a unique local minimizer/multiplier pair.
Note that in some cases, the subproblem may have a unique local minimizer, even though the original problem has multiple local minima.
See also \S\ref{sec:discussion} for further discussion.
\begin{assumption}
\label{assm:nearestSolution}
For all iterations $k$, 
$(\x^{k+1},\bmu^{k+1})$ 
is the local minimizer/multiplier for subproblem~\eqref{ADMMSP} that is closest (in 2-norm distance) to
$(\x^*,\bmu^*)$.
\end{assumption}

To help clarify some discussion, the following assumption states regularity conditions on the subproblem similar to the overall conditions in Assumption~\ref{assm:OverallRegularity}.
In Appendix~\ref{app:RegularityRelation}, it is shown that Assumption~\ref{assm:OverallRegularity} implies Assumption~\ref{assm:SPRegularity}, given that $\rho$ is sufficiently large.
Consequently, we will not need the following assumption explicitly in the main result; however it is useful in some of the supporting results.
\begin{assumption}
\label{assm:SPRegularity}
Assume that
$f$ and $\c$ are twice continuously differentiable and that
$(\x^*, \y^*, \bmu^*, \blam^*)$ is a KKT point of Problem~\eqref{problem}.
Assume that
$(\x^*, \bmu^*)$ 
is a KKT point of Subproblem~\eqref{ADMMSP} for 
$\y^k = \y^*$ and $\blam^k = \blam^*$.
%
Assume that $\x^*$ satisfies the linear independence constraint qualification:
the vectors 
$\set{\grad c_j(\x^*) : j \in \set{1,\dots,p} }$
are linearly independent.
Assume that the second order sufficient conditions hold:
\[
\z \tr 
\Big( \grad^2 f(\x^*)  + \smallsum_j \mu_{j}^* \grad^2 c_{j}(\x^*) + \rho \mbf{A} \tr \mbf{A} \Big) 
\z > 0
\]
for all $\z$ satisfying
$\z \neq \0$, 
$\grad c_{j}(\x^*) \tr \z = 0$ for all $j$.
\end{assumption}

\subsection{Preliminary analysis}
\label{sec:prelim}
The main convergence analysis is given in the following section.
This section includes some definitions and preliminary analysis that is useful for the convergence result.

Under Assumption~\ref{assm:KKTiterates}, we have that for each $k$,
there exists $\bmu^{k+1}$  such that 
$(\x^{k+1},\bmu^{k+1})$
is a KKT point of subproblem~\eqref{ADMMSP}.
%
Rearranging and adding $\rho \mbf{A} \tr \mbf{B}\y^{k+1}$ to both sides of Equation~\eqref{spKKT1} of the subproblem KKT conditions gives 
\begin{equation}
\notag
\grad f(\x^{k+1})
	+ \grad \c(\x^{k+1}) \bmu^{k+1}
	+ \mbf{A} \tr (\blam^k + \rho (\mbf{A}\x^{k+1} +\mbf{B}\y^{k+1} - \b))
	= \rho \mbf{A} \tr \mbf{B}(\y^{k+1} - \y^k).
\end{equation}
Using the update formulas for $\blam^{k+1}$ and the dual residual $\r^{k+1}$ from Algorithm~\ref{alg:ADMM}, we have for all $k$
\begin{equation}
\label{eq:toGetKKT}
\grad f(\x^{k+1})
	+ \grad \c(\x^{k+1}) \bmu^{k+1}
	+ \mbf{A} \tr \blam^{k+1}
	= \r^{k+1}.
\end{equation}
Comparing the above with Equation~\eqref{mainKKTx}, we see that the gradient with respect to $\x$ of the Lagrangian of Problem~\eqref{problem} evaluated at  
$(\x^{k+1},\bmu^{k+1},\blam^{k+1})$
is $\r^{k+1}$,
which makes its definition as the dual residual appropriate.
Further, it is clear that $\q^{k+1}$ equals the violation of the feasibility condition~\eqref{mainKKTfeas2}.
Finally, note that as
$(\x^{k+1},\y^{k+1},\bmu^{k+1},\blam^{k+1})$
approaches a KKT point
$(\x^*,\y^*,\bmu^*,\blam^*)$,
$\q^{k+1}$ goes to zero, and by continuous differentiability under Assumption~\ref{assm:KKTiterates}, $\r^{k+1}$ goes to zero as well. 

Meanwhile, note that we may solve analytically for $\y^{k+1}$.
First, note that the problem for $\y^{k+1}$ is convex, so there is no need to distinguish between whether we find a local or global minimizer.
The first order optimality conditions are
\begin{equation}
\label{eq:y_stationarity}
	\mbf{B}\tr \blam^k + \rho\mbf{B}\tr(\mbf{A}\x^{k+1} +\mbf{B}\y^{k+1} - \b) = \0.
\end{equation}
Using the update rule for $\blam^{k+1}$, this means
\[ 
\mbf{B} \tr \blam^{k+1} = \0.
\]
This helps explain the requirement that $\mbf{B} \tr \blam^0 = \0$;
we can then assume that for all $k$, $\mbf{B} \tr \blam^k = \0$.
Then assuming that $\mbf{B}$ has full column rank (as would hold under Assumption~\ref{assm:OverallRegularity}), $\mbf{B}\tr\mbf{B}$ is invertible, and we can transform Equation~\eqref{eq:y_stationarity} to obtain
\begin{equation}
\label{eq:y_update}
 \y^{k+1} = (\mbf{B}\tr\mbf{B})\inv \mbf{B}\tr(\b - \mbf{A}\x^{k+1}).
\end{equation}
This formula also helps us interpret $\y^{k+1}$ as the least-squares solution to satisfying the constraints $\mbf{A}\x^{k+1} + \mbf{B}\y = \b$.

\section{Convergence analysis}
\label{sec:localConv}

In this section we explore the local convergence of Algorithm~\ref{alg:ADMM}.
The first subsection builds a number of results required for the main convergence result, and in so doing outlines the analytical approach taken.
The second subsection states and proves the main local convergence result, Theorem~\ref{thm:convergence}, as well as a corollary.

\subsection{Supporting results}

A key observation that enables the analysis is that $(\x^{k+1},\y^{k+1},\bmu^{k+1},\blam^{k+1})$ can be identified with a local optimal solution of a perturbed version of Problem~\eqref{problem}, Problem~\eqref{ProbFamily} below,
when $(\r,\q) = (\r^{k+1},\q^{k+1})$;
compare with, for instance, Equation~\eqref{eq:toGetKKT}.
Using sensitivity analysis, we can show that $(\x^{k+1},\y^{k+1},\bmu^{k+1},\blam^{k+1})$ approaches the optimal value $(\x^*,\y^*,\bmu^*,\blam^*)$ as the residuals go to zero.
Consequently, we use an appropriate Lyapunov function to show that the residuals do indeed go to zero.
While some of this analysis is similar in overall structure to the basic convergence result for ADMM given in \cite[Appendix A]{boydEA11}, the introduction of the perturbed problem~\eqref{ProbFamily} is a novel technical detail required to adapt it to the nonconvex case.

We start by analyzing the perturbed version of Problem~\eqref{problem}.
The main goal of the following lemma is to derive properties of the resulting ``primal functional.''
\begin{lemma}
\label{lem:primalProperties}
Let Assumption~\ref{assm:OverallRegularity} hold.
Consider the family of problems parameterized by 
$(\r,\q) \in \mbb{R}^n \times \mbb{R}^q$:
\begin{equation}
\label{ProbFamily}
\min_{\x, \y}
\set{ 
f(\x) - \r \tr \x : 
\c(\x) = \0, \;
\mbf{A}\x  + \mbf{B}\y = \b + \q
}.
\end{equation}
\begin{enumerate} \itemsep0pt \parskip0pt \parsep0pt
\item
There exists a positive constant $\epsilon_1$ and continuously differentiable functions 
$\hat\x$, $\hat\y$, $\hat\bmu$, $\hat\blam$
such that for all $\r$, $\q$ with $\norm{(\r,\q)} < \epsilon_1$,
$
(\hat\x(\r,\q), \hat\y(\r,\q), \hat\bmu(\r,\q), \hat\blam(\r,\q))
$
is a KKT point of \eqref{ProbFamily}.
It holds that 
$(\hat\x(\0,\0), \hat\y(\0,\0), \hat\bmu(\0,\0), \hat\blam(\0,\0)) = (\x^*,\y^*,\bmu^*,\blam^*)$.
Furthermore, 
$(\hat\x, \hat\y, \hat\bmu, \hat\blam)$
is unique, in the sense that there exists $\epsilon_1' > 0$ such that if
$(\tilde{\x},\tilde{\y},\tilde{\bmu},\tilde{\blam})$
is a KKT point of \eqref{ProbFamily} and
$\norm{(\tilde{\x},\tilde{\y},\tilde{\bmu},\tilde{\blam}) - (\x^*,\y^*,\bmu^*,\blam^*)} < \epsilon_1'$, 
then 
$(\tilde{\x},\tilde{\y},\tilde{\bmu},\tilde{\blam}) = 
(\hat\x(\r,\q), \hat\y(\r,\q), \hat\bmu(\r,\q), \hat\blam(\r,\q))$.
\item
There exist positive constants $\epsilon_2$ and $\epsilon_2'$ such that 
for all $(\r,\q)$ with $\norm{(\r,\q)} < \epsilon_2$, 
we have that $(\hat\x(\r,\q),\hat\y(\r,\q))$ is a local minimizer of \eqref{ProbFamily}, on a neighborhood with radius $\epsilon_2'$.
%
\item
There exist positive constants $\rho'$ and $\epsilon_3$ such that for all $(\r,\q)$ with $\norm{(\r,\q)} < \epsilon_3$ and $\rho > \rho'$, we have
\[
f(\x^*)          - \r \tr \x^* \ge 
f(\hat\x(\r,\q)) - \r \tr \hat\x(\r,\q)
	- \frac{\rho}{8}\norm{\q}^2 + (\hat{\blam}(\r,\q))\tr \q.
\]
\end{enumerate}
\end{lemma}
\begin{proof}~
\begin{enumerate}
\item 
This claim is a fairly standard sensitivity result; 
see for instance \cite[Thm.~5.1]{fiaccoEA90} or \cite[Thm.~2.1]{fiacco76}.
The specific claim of uniqueness is not often stated so explicitly, however it follows from the implicit function theorem upon which the result is based
(see for instance \S1.2 of \cite{bertsekas_lmm} or \cite[Prop.~A.25]{bertsekas_nlp}).
\item
%
This claim is a statement that there is a ``smallest'' neighborhood on which
$(\hat\x(\r,\q),\hat\y(\r,\q))$ 
is a local minimum, for all sufficiently small $(\r,\q)$.
Let the fully augmented Lagrangian of \eqref{ProbFamily} 
(at the optimal multipliers $(\hat\bmu,\hat\blam)$) be
\[
\begin{aligned}
{L}_{\rho,\r,\q} : 
	(\x,\y) \mapsto  
	f(\x) - \r\tr\x
	&+ \hat\bmu(\r,\q) \tr \c(\x)
	 + \frac{\rho}{2}\norm{\c(\x)}^2 \\
	&+ \hat\blam(\r,\q) \tr (\mbf{A}\x+\mbf{B}\y-\b-\q)
	 + \frac{\rho}{2}\norm{\mbf{A}\x+\mbf{B}\y-\b-\q}^2.
\end{aligned}
\]
Recall the definition of $\mbf{H}_{xx}$ from Assumption~\ref{assm:OverallRegularity} as the Hessian of the Lagrangian of Problem~\eqref{problem}.
Then the Hessian of the augmented Lagrangian $L_{\rho,\r,\q}$ above is given by%
\footnote{The gradient of the augmented Lagrangian is given by
\[
\begin{aligned}
&\grad_x {L}_{\rho,\r,\q}(\x,\y) = 
	   \grad f(\x) - \r
	 + \grad \c(\x) \hat\bmu(\r,\q) 
	 + \rho \grad \c(\x) \c(\x)
	 + \mbf{A}\tr \hat\blam(\r,\q)
	 + \rho \mbf{A}\tr (\mbf{A}\x + \mbf{B}\y - \b -\q), \\
&\grad_y {L}_{\rho,\r,\q}(\x,\y) = 
	  \mbf{B}\tr \hat\blam(\r,\q)
	+ \rho \mbf{B}\tr (\mbf{A}\x + \mbf{B}\y - \b -\q).
\end{aligned}
\]}
\[
\begin{aligned}
\grad^2_{xx} {L}_{\rho,\r,\q}(\x) &=
	\mbf{H}_{xx}(\x,\hat\bmu(\r,\q)) 
	+ \smallsum_j \rho c_j(\x) \grad^2 c_j(\x)
	+ \rho \grad \c(\x) \grad \c(\x) \tr
	+ \rho \mbf{A}\tr \mbf{A}, \\
\grad^2_{yy} {L}_{\rho,\r,\q}(\x) &=
	\rho \mbf{B}\tr\mbf{B}, \\
\grad^2_{yx} L_{\rho,\r,\q}(\x) &=
	\rho \mbf{B} \tr \mbf{A}
	= \grad^2_{xy} L_{\rho,\r,\q}(\x) \tr.
\end{aligned}
\]
Then define
\[
\mbf{H}_{\rho} : (\x,\r,\q) \mapsto 
\grad^2 L_{\rho,\r,\q}(\x) = 
\begin{bmatrix}
\grad^2_{xx} {L}_{\rho,\r,\q}(\x) & \grad^2_{xy} {L}_{\rho,\r,\q}(\x) \\ 
\grad^2_{yx} {L}_{\rho,\r,\q}(\x) & \grad^2_{yy} {L}_{\rho,\r,\q}(\x)
\end{bmatrix}
\]
where we highlight its functional dependence on $(\r,\q)$.
Note that $\mbf{H}_{\rho}$ is continuous with respect to $(\x,\r,\q)$ under Assumption~\ref{assm:OverallRegularity}, since the defining functions are twice continuously differentiable and $\hat\bmu$ is continuous.
Since $\c(\x^*) = \0$, note that 
\[
\mbf{H}_{\rho}(\x^*,\0,\0) = 
	\begin{bmatrix}
	\mbf{H}_{xx}(\x^*,\bmu^*) & \0 \\
	\0 & \0
	\end{bmatrix}
	 + 
	\rho \mbf{C} \tr \mbf{C}
\]
recalling that $\mbf{C}$ is the Jacobian of the equality constraints at $\x^*$.
Under the second-order sufficient conditions of Assumption~\ref{assm:OverallRegularity}, 
$\mbf{H}_{\rho}(\x^*,\0,\0)$
is positive definite for some $\rho > 0$, by, for instance, \cite[Lemma~3.2.1]{bertsekas_nlp}.
Combined with the continuity of $\hat\x$ and $\mbf{H}_{\rho}$, 
we can choose $\epsilon_p > 0$ such that $\mbf{H}_{\rho}(\hat\x(\r,\q),\r,\q)$ is positive definite for all $(\r,\q)$ such that 
$\norm{(\r,\q)} < \epsilon_p$.
Since $(\hat\x,\hat\y,\hat\bmu,\hat\blam)$ is a KKT point for problem~\eqref{ProbFamily}, we note that
$\grad L_{\rho,\r,\q} (\hat\x(\r,\q),\hat\y(\r,\q)) = \0$
for all $(\r,\q)$. 
Consequently, we can apply Lemma~\ref{lem:MinimizerRadius} in Appendix~\ref{sec:technicalLemma} to see that there exist positive $\epsilon_2$ and $\epsilon_2'$ such that 
for all $(\r,\q)$ with $\norm{(\r,\q)} < \epsilon_2$,
we have
$(\hat\x(\r,\q),\hat\y(\r,\q))$
is a minimizer of $L_{\rho,\r,\q}$ on the neighborhood 
$\set{ (\x,\y) : \norm{(\x,\y) - (\hat\x(\r,\q),\hat\y(\r,\q))} < \epsilon_2' }$.

Finally, for all $\x,\y$ such that 
$\mbf{A}\x + \mbf{B}\y = \b + \q$ and 
$\c(\x) = \0$,
$L_{\rho,\r,\q}(\x,\y) = f(\x) - \r \tr \x$,
and so it follows that $(\hat\x(\r,\q),\hat\y(\r,\q))$ is a local minimizer of Problem~\eqref{ProbFamily} on an $\epsilon_2'$-neighborhood, for all $(\r,\q)$ such that $\norm{(\r,\q)} < \epsilon_2$.

\item
This claim uses the fact that the ``penalized primal functional'' is convex.
Let $p$ be the primal functional (the optimal objective value) of \eqref{ProbFamily};
i.e., it is defined by 
$p : (\r,\q) \mapsto f(\hat\x(\r,\q)) - \r \tr \hat\x(\r,\q)$.
Again, from standard sensitivity analysis we have that the gradient of $p$ with respect to $\q$ is
$\grad_{q} p(\r,\q) = -\hat\blam(\r,\q)$
(see \cite[Proposition~3.3.3]{bertsekas_nlp}).
Since $\hat\blam$ is continuously differentiable, we have that $\grad_{qq}^2 p$ is continuous.
By \cite[Lemma 3.2.1]{bertsekas_nlp}, we can choose $\bar\rho > 0$ so that
\[
\grad_{qq}^2 p(\0,\0) + \bar\rho \mbf{I} \symgt \0.
\]
Let 
$S_{\bar\rho} = \set{ (\r,\q) : \grad_{qq}^2 p(\r,\q) + \bar\rho \mbf{I} \symgt \0 }$
which, by the continuity of $\grad_{qq}^2 p$, is open and contains $(\0,\0)$.
If 
$\grad_{qq}^2 p(\r,\q) + \bar\rho \mbf{I} \symgt \0$, then 
$\grad_{qq}^2 p(\r,\q) +     \rho \mbf{I} \symgt \0$ for any $\rho > \bar\rho$.
Thus $S_{\bar\rho} \subset S_{\rho}$ for all $\rho > \bar\rho$.
Thus we can choose $\epsilon_3 > 0$ so that 
$\grad_{qq}^2 p(\r,\q) + \sfrac{\rho}{4} \mbf{I}$ 
is positive definite for all $(\r,\q)$ and $\rho$ such that $\norm{(\r,\q)} < \epsilon_3$ and $\rho > 4\bar\rho$.
It follows that for all sufficiently small $\r$ and sufficiently large $\rho$,
$\q \mapsto p(\r,\q) + \frac{\rho}{8}\norm{\q}^2$
is convex on the set of $\q$ such that $\norm{(\r,\q)} < \epsilon_3$.

Next, the gradient of 
$\q \mapsto p(\r,\q) + \frac{\rho}{8}\norm{\q}^2$ 
is
$-\hat{\blam}(\r,\q) + \frac{\rho}{4} \q$.
For a convex function, a gradient is a subgradient and so
\[
p(\r,\mbf{0}) + \frac{\rho}{8}\norm{\mbf{0}}^2 \ge 
p(\r,\q) +      \frac{\rho}{8}\norm{\q}^2 + (-\hat{\blam}(\r,\q) + \frac{\rho}{4} \q) \tr (\mbf{0}-\q)
\]
and so for all $\q$ such that $\norm{(\r,\q)} < \epsilon_3$,
\begin{align*}
p(\r,\mbf{0}) 
&\ge p(\r,\q) + \frac{\rho}{8}\norm{\q}^2 - \frac{\rho}{4} \norm{\q}^2 + \hat{\blam}(\r,\q) \tr \q \\
&=   p(\r,\q) - \frac{\rho}{8}\norm{\q}^2 + (\hat{\blam}(\r,\q))\tr \q.
\end{align*}
Then, using part 2, take $\r$ small enough that we have 
$\norm{(\r,\0)} < \epsilon_2$ and
$\norm{(\hat\x(\r,\0),\hat\y(\r,\0)) - (\x^*,\y^*)} < \epsilon_2'$.
Then $(\x^*,\y^*)$ is feasible in \eqref{ProbFamily} (for $\q = \0$) and in the neighborhood on which $(\hat\x(\r,\0),\hat\y(\r,\0))$ is a minimizer, and so $(\x^*,\y^*)$ must have greater or equal objective value.
Combining this with the inequality above yields the claim, defining $\rho' = 4\bar\rho$ and $\epsilon_3$ as necessary.
\end{enumerate}
\end{proof}

Next we show two inequalities, which provide bounds on the difference between the objective value  at the solution $(\x^*,\y^*)$ and iterates of the algorithm.

\begin{lemma}
\label{lem:step2}
Let Assumptions~\ref{assm:KKTiterates} and \ref{assm:OverallRegularity} hold.
There exist positive $\epsilon$ and $\rho'$ such that,
if
$\rho > \rho'$
and
$\norm{(\x^{k+1},\y^{k+1},\bmu^{k+1},\blam^{k+1}) - (\x^*,\y^*,\bmu^*,\blam^*)} < \epsilon$,
then
\begin{equation}
\label{ineq:2}
f(\x^{k+1}) - f(\x^*)
	\le 
  (\r^{k+1}) \tr (\x^{k+1} - \x^*)
- (\blam^{k+1}) \tr \q^{k+1}
+ \frac{\rho}{8}\norm{\q^{k+1}}^2.
\end{equation}
\end{lemma}
\begin{proof}
From Equation~\eqref{eq:toGetKKT}, we have
\[
\grad f(\x^{k+1})
	+ \grad \c(\x^{k+1}) \bmu^{k+1}
	+ \mbf{A} \tr \blam^{k+1}
	- \r^{k+1} 
= \0.
\]
Combined with
$\c(\x^{k+1}) = \0$ and 
$\mbf{B}\tr \blam^{k+1} = \0$,
this implies that 
$(\x^{k+1},\y^{k+1}, \bmu^{k+1}, \blam^{k+1})$
is a KKT point of 
\begin{equation}
\label{someProb}
\min \set{ 
f(\x) - (\r^{k+1})\tr \x : 
	\c(\x) = \0,
	\mbf{A}\x + \mbf{B}\y = \b + \q^{k+1}
}.
\end{equation}

However, by Lemma~\ref{lem:primalProperties}, we know that for
$\r^{k+1}$ and $\q^{k+1}$ sufficiently close to zero,
Problem~\eqref{someProb} has a KKT point 
\[
(	\hat\x(\r^{k+1},\q^{k+1}), \hat\y(\r^{k+1},\q^{k+1}), 
	\hat\bmu(\r^{k+1},\q^{k+1}), \hat\blam(\r^{k+1},\q^{k+1}) )
\]
which is unique in a neighborhood of 
$(\x^*,\y^*,\bmu^*,\blam^*)$.
Consequently, for 
$(\x^{k+1},\y^{k+1}, \bmu^{k+1},\blam^{k+1})$
sufficiently close to 
$(\x^*,\y^*,\bmu^*,\blam^*)$,
we can conclude from part~1 of Lemma~\ref{lem:primalProperties} that in fact
\begin{equation}
\label{eq:ParameterizedIterate}
(\x^{k+1}, \y^{k+1}, \bmu^{k+1},\blam^{k+1}) = 
(	\hat\x(\r^{k+1},\q^{k+1}), \hat\y(\r^{k+1},\q^{k+1}), 
	\hat\bmu(\r^{k+1},\q^{k+1}), \hat\blam(\r^{k+1},\q^{k+1}) ).
\end{equation}
Thus, from part 3 of Lemma~\ref{lem:primalProperties}, there exist positive $\rho'$ and $\epsilon_3$ such that for $\rho > \rho'$ and $\norm{(\r^{k+1},\q^{k+1})} < \epsilon_3$,
\begin{equation*}
f(\x^{*}) -   (\r^{k+1})\tr \x^{*} \ge
f(\x^{k+1}) - (\r^{k+1})\tr \x^{k+1}
	 - \frac{\rho}{8}\norm{\q^{k+1}}^2 + (\blam^{k+1})\tr \q^{k+1}.
\end{equation*}
Rearranging yields the desired inequality.
Noting that $\r^{k+1}$ and $\q^{k+1}$ go to zero as 
$(\x^{k+1},\y^{k+1},\bmu^{k+1},\blam^{k+1})$ 
approaches the optimal value,
we can take $\epsilon$ sufficiently small to ensure 
$\norm{(\r^{k+1},\q^{k+1})} < \epsilon_3$.
This yields the result.
\end{proof}

\begin{lemma}
\label{lem:step1}
Let Assumption~\ref{assm:OverallRegularity} hold.
There exist positive $\rho''$ and $\epsilon$ such that,
if $\rho > \rho''$ and 
$\norm{(\x^{k+1},\y^{k+1}) - (\x^*,\y^*)} < \epsilon$,
then
\begin{equation}
\label{ineq:1}
f(\x^*) - f(\x^{k+1}) \le (\blam^*) \tr \q^{k+1} + \frac{\rho}{8}\norm{\q^{k+1}}^2.
\end{equation}
\end{lemma}
\begin{proof}
This follows from arguments about the augmented Lagrangian, similarly to the proof of Lemma~\ref{lem:primalProperties}, part 2.
Let 
\[
L_{\rho} : (\x,\y) \mapsto 
	f(\x) + 
	(\bmu^*) \tr \c(\x) 	     			+ \frac{\rho}{2}\norm{\c(\x)}^2 + 
	(\blam^*)\tr (\mbf{A}\x+\mbf{B}\y-\b)	+ \frac{\rho}{8}\norm{\mbf{A}\x+\mbf{B}\y-\b}^2
\]
(the factor of $\sfrac{1}{8}$ on the penalty term is deliberate, and will be used later).
Then by the KKT necessary conditions for Problem~\eqref{problem},
$\grad L_{\rho} (\x^*,\y^*) = \0$,
and by the second order sufficient conditions, for $\rho$ sufficiently large,  
$\grad^2 L_{\rho} (\x^*,\y^*)$
is positive definite.
Similarly to the proof of \cite[Prop.~1.1.3]{bertsekas_nlp},
$(\x^*,\y^*)$ is a minimizer of $L_{\rho}$ on some neighborhood, and the radius of this neighborhood is independent%
\footnote{The radius of this neighborhood depends on the minimum eigenvalue of 
$\grad^2 L_{\rho} (\x^*,\y^*)$.
In particular, the radius is non-decreasing as this minimum eigenvalue increases.
While $\grad^2 L_{\rho} (\x^*,\y^*)$ does depend on the value of $\rho$, the minimum eigenvalue can only increase with increasing $\rho$;
compare with the expression for  $\mbf{H}_{\rho}$ in the proof of Lemma~\ref{lem:primalProperties}, part~2.
Thus, the radius of the neighborhood on which $\x^*$ is a minimizer is independent of $\rho$, as long as $\rho$ is above the critical value.}
of $\rho$.
Consequently, there exist positive constants $\epsilon$ and $\rho''$ 
such that for $\norm{(\x,\y) - (\x^*,\y^*)} < \epsilon$ and $\rho > \rho''$, we have
\[
L_{\rho}(\x^*,\y^*) \le 
L_{\rho}(\x,\y).
\]
Since $\mbf{A}\x^* +\mbf{B}\y^* - \b = \c(\x^*) = \0$, we have
$L_{\rho}(\x^*,\y^*) = f(\x^*)$.
Finally, since $\x^{k+1}$ satisfies $\c(\x^{k+1}) = \0$, it holds that
\[
f(\x^*) \le
f(\x^{k+1}) 
+ (\blam^*)\tr(\mbf{A}\x^{k+1} + \mbf{B}\y^{k+1} - \b) 
+ \frac{\rho}{8} \norm{\mbf{A}\x^{k+1} + \mbf{B}\y^{k+1} - \b}^2.
\]
Rearranging the above and using the definition of the primal residual yields the desired inequality.
\end{proof}

Similar to the convergence proof in \cite[Appendix A]{boydEA11} and \cite{chatzipanagiotisEA17}, we define a Lyapunov function
\begin{equation}
\label{eq:lyapunov}
V^k = \frac{1}{\rho}\norm{\blam^k - 	\blam^*}^2 + \rho\norm{\mbf{B}(\y^k - \y^*)}^2.
\end{equation}
Using Lemmata~\ref{lem:step2} and \ref{lem:step1}, the following proposition asserts the existence of a neighborhood around the optimal point 
$(\x^*,\y^*,\bmu^*,\blam^*)$
so that if the iterates of the algorithm fall in this neighborhood, then we obtain a bound on the decrease in the Lyapunov function.

\begin{proposition}
\label{prop:LyapunovDecreaseBound}
Let Assumptions~\ref{assm:KKTiterates} and \ref{assm:OverallRegularity} hold.
There exist positive constants $\epsilon$ and $\rho^*$ such that,
if
$\rho > \rho^*$ and
$\norm{(\x^{k+1},\y^{k+1},\bmu^{k+1},\blam^{k+1}) - (\x^*,\y^*,\bmu^*,\blam^*)} < \epsilon$,
for some $k$,
then it holds that
\begin{equation}
\label{ineq:LyapunovDecreaseBound}
0 \le 
V^k - V^{k+1}
-	\rho\norm{\mbf{B}(\y^{k} - \y^{k+1})}^2
-	\frac{\rho}{2}\norm{\q^{k+1}}^2.
\end{equation}
\end{proposition}
\begin{proof}
Defining $\rho^*$ and $\epsilon$ as necessary, we can apply Lemmata~\ref{lem:step2} and \ref{lem:step1}.
Upon adding Inequalities~\eqref{ineq:2} and \eqref{ineq:1} and multiplying by two, we obtain
\begin{equation}
\label{ineq:Intermediate1}
0 \le
  2(\x^{k+1} - \x^*)   \tr \r^{k+1}
+ 2(\blam^* - \blam^{k+1}) \tr \q^{k+1}
+ 2\frac{\rho}{4}\norm{\q^{k+1}}^2.
\end{equation}

Using the definition of the dual residual, the first term in \eqref{ineq:Intermediate1} is
\[
\begin{aligned}
2(\x^{k+1} - \x^*)   \tr \r^{k+1}
&=  2(\x^{k+1} - \x^*)   \tr (\rho\mbf{A}\tr\mbf{B}(\y^{k+1} - \y^k)) \\
&=  2\rho (\y^{k+1} - \y^k)\tr \mbf{B}\tr (\mbf{A}\x^{k+1} - \mbf{A}\x^*) \\
&=  2\rho (\y^{k+1} - \y^k)\tr (\mbf{B}\tr\mbf{B}) (\y^* - \y^{k+1})
\end{aligned}
\]
where we have used Formula~\eqref{eq:y_update} to see that
$\mbf{B}\tr\mbf{B} \y^{k+1} = \mbf{B}\tr(\b - \mbf{A}\x^{k+1})$, and noting that
$\mbf{A}\x^* + \mbf{B}\y^* = \b$, and so 
$\mbf{B}\tr\mbf{B} \y^{*} = \mbf{B}\tr(\b - \mbf{A}\x^{*})$
(and recall that Assumption~\ref{assm:OverallRegularity} implies that $\mbf{B}$ has full column rank, validating the use of Formula~\eqref{eq:y_update}).
Then note that
\[
2\rho (\y^{k+1} - \y^k)\tr (\mbf{B}\tr\mbf{B}) (\y^* - \y^{k+1})
=
  \rho\norm{\mbf{B}(\y^k - \y^*)}^2 
- \rho\norm{\mbf{B}(\y^{k+1} - \y^*)}^2 
- \rho\norm{\mbf{B}(\y^{k} - \y^{k+1})}^2,
\]
which is seen after expanding out both sides.

Using 
$\q^{k+1} = \frac{1}{\rho}(\blam^{k+1} - \blam^k)$,
the second term in the right-hand side of \eqref{ineq:Intermediate1} is
\begin{equation*}
2(\blam^* - \blam^{k+1}) \tr \q^{k+1}
= \frac{2}{\rho}(\blam^* - \blam^{k+1})\tr(\blam^{k+1} - \blam^{k}).
\end{equation*}
Then note that
\begin{align*}
\frac{2}{\rho}(\blam^* - \blam^{k+1})\tr(\blam^{k+1} - \blam^{k})
=	\frac{1}{\rho}\norm{\blam^k - 		\blam^*}^2
 - \frac{1}{\rho}\norm{\blam^{k+1} - 	\blam^*}^2 
 - \frac{1}{\rho}\norm{\blam^{k+1} - 	\blam^k}^2,
\end{align*}
which, again, is seen after expanding out both sides.
Then using $\blam^{k+1} - \blam^k = \rho \q^{k+1}$ in the right-hand side of the expression above, we can combine with the third term in the right-hand side of \eqref{ineq:Intermediate1} so that
\[
2(\blam^* - \blam^{k+1}) \tr \q^{k+1}
+ 2\frac{\rho}{4}\norm{\q^{k+1}}^2
=
	\frac{1}{\rho}\norm{\blam^k - 		\blam^*}^2
 - \frac{1}{\rho}\norm{\blam^{k+1} - 	\blam^*}^2 
 - \frac{\rho}{2}\norm{\q^{k+1}}^2.
\]

Consequently, Inequality~\eqref{ineq:Intermediate1} becomes
\begin{align}
\notag
0 \le &
	\left( \frac{1}{\rho}\norm{\blam^k - 	\blam^*}^2 + 	\rho\norm{\mbf{B}(\y^k - \y^*)}^2 \right)
-	\left( \frac{1}{\rho}\norm{\blam^{k+1}-\blam^*}^2 +	\rho\norm{\mbf{B}(\y^{k+1} - \y^*)}^2 \right)
\\
& \notag
-	\rho\norm{\mbf{B}(\y^{k} - \y^{k+1})}^2
-	\frac{\rho}{2}\norm{\q^{k+1}}^2,
\end{align}
which is the desired conclusion.
\end{proof}

We iterate Inequality~\eqref{ineq:LyapunovDecreaseBound} to show that the residuals converge to zero.

\begin{proposition}
\label{prop:PrelimLocalConv}
Let Assumptions~\ref{assm:KKTiterates} and \ref{assm:OverallRegularity} hold.
There exist positive constants $\epsilon$ and $\rho^*$ such that,
if 
$\rho > \rho^*$
and 
$\norm{(\x^{k+1},\y^{k+1},\bmu^{k+1},\blam^{k+1}) - (\x^*,\y^*,\bmu^*,\blam^*)} < \epsilon$,
for all sufficiently large $k$,
then the sequence
\[
\seq[k \in \mbb{N}]{ \rho\norm{\mbf{B}(\y^{k} - \y^{k+1})}^2 +	\frac{\rho}{2}\norm{\q^{k+1}}^2 }
\]
converges to zero.
\end{proposition}
\begin{proof}
From Proposition~\ref{prop:LyapunovDecreaseBound}, we know that there exist constants $\epsilon$ and $\rho^*$ such that,
if
$\rho > \rho^*$
and 
$\norm{(\x^{k+1},\y^{k+1},\bmu^{k+1},\blam^{k+1}) - (\x^*,\y^*,\bmu^*,\blam^*)} < \epsilon$,
then Inequality~\eqref{ineq:LyapunovDecreaseBound} holds.
Thus, if the conditions hold for all sufficiently large $k$, then there exists $K$ such that
\[
V^{k+1}
+	\rho \norm{\mbf{B}(\y^{k} - \y^{k+1})}^2
+	\frac{\rho}{2} \norm{\q^{k+1}}^2 
\le
V^k,
\]
for all $k \ge K$.
Iterating we get
\[
V^{K+\ell} + 
\sum_{k = K}^{K + \ell - 1} \left( \rho \norm{\mbf{B}(\y^{k} - \y^{k+1})}^2 + \frac{\rho}{2} \norm{\q^{k+1}}^2 \right)
\le
V^K
\]
for all $\ell \ge 1$.
The partial sums in this expression are bounded above and increasing, 
since $V^{K+\ell}$ is always nonnegative, $V^K$ is finite, and the terms in the sum above are always nonnegative.
It follows that the partial sums converge as $\ell \to \infty$, and thus that
$\seq[k]{ \rho\norm{\mbf{B}(\y^{k} - \y^{k+1})}^2 +	\frac{\rho}{2}\norm{\q^{k+1}}^2 }$
converges to zero.
\end{proof}


The following preliminary convergence result combines the observation from Lemma~\ref{lem:primalProperties} that the iterates $(\x^{k+1},\y^{k+1},\bmu^{k+1},\blam^{k+1})$ depend continuously on the residuals and the observation from  Proposition~\ref{prop:PrelimLocalConv} that the residuals converge to zero.
However, note that the result only really asserts the \emph{existence} of a neighborhood around the optimal point $(\x^*,\y^*,\bmu^*,\blam^*)$
at which convergence can occur.

\begin{proposition}
\label{prop:almostConvergence}
Let Assumptions~\ref{assm:KKTiterates} and \ref{assm:OverallRegularity} hold.
There exist positive constants $\epsilon$ and $\rho^*$ such that,
if 
$\rho > \rho^*$ 
and
$\norm{(\x^{k+1},\y^{k+1},\bmu^{k+1},\blam^{k+1}) - (\x^*,\y^*,\bmu^*,\blam^*)} < \epsilon$
for all sufficiently large $k$,
then the sequence of iterates produced by Algorithm~\ref{alg:ADMM}
$
\seq[k \in \mbb{N}]{(\x^{k+1},\y^{k+1},\bmu^{k+1},\blam^{k+1})}
$
converges to $(\x^*,\y^*,\bmu^*,\blam^*)$.
\end{proposition}
\begin{proof}
By Proposition~\ref{prop:PrelimLocalConv}, we have 
$\seq[k]{ \rho \norm{\mbf{B}(\y^{k} - \y^{k+1})}^2 }$ and
$\seq[k]{ \frac{\rho}{2}\norm{\q^{k+1}}^2  }$
converging to zero;
this implies $\seq[k]{ \r^{k+1} }$ converges to zero and clearly that $\seq[k]{ \q^{k+1} }$ converges to zero.
Then, using Lemmata~\ref{lem:primalProperties} and \ref{lem:step2}, specifically Equation~\eqref{eq:ParameterizedIterate}, we note that 
$(\x^{k+1},\y^{k+1},\bmu^{k+1},\blam^{k+1})$ equals 
\[
\big( \hat\x(\r^{k+1},\q^{k+1}), \hat\y(\r^{k+1},\q^{k+1}),
	  \hat\bmu(\r^{k+1},\q^{k+1}),\hat\blam(\r^{k+1},\q^{k+1}) \big),
\] 
where $(\hat\x,\hat\y,\hat\bmu,\hat\blam)$ is a continuous function equaling $(\x^*,\y^*,\bmu^*,\blam^*)$ at $(\0,\0)$.
Combined with the convergence of the residuals to zero, we have the result.
\end{proof}

To obtain a stronger result, we begin to invoke Assumption~\ref{assm:nearestSolution}.
The following result takes advantage of Assumption~\ref{assm:nearestSolution} to show that the iterates $(\x^{k+1},\bmu^{k+1})$ remain close to $(\x^*,\bmu^*)$ if $(\y^k,\blam^k)$ remains close to $(\y^*,\blam^*)$.
\begin{lemma}
\label{lem:continuousXMU}
Let Assumptions~\ref{assm:nearestSolution} and \ref{assm:SPRegularity} hold.
There exist continuous functions
$\x^+$, $\bmu^+$
defined on a neighborhood of $(\y^*,\blam^*)$
such that
\[\begin{aligned}
(\x^+(\y^k,\blam^k),\bmu^+(\y^k,\blam^k)) &= (\x^{k+1},\bmu^{k+1}), \\
(\x^+(\y^*,\blam^*),\bmu^+(\y^*,\blam^*)) &= (\x^*,\bmu^*).
\end{aligned}\]
%
\end{lemma}
\begin{proof}
Under Assumption~\ref{assm:SPRegularity}, we can apply standard sensitivity analysis results, such as \cite[Thm.~5.1]{fiaccoEA90} or \cite[Thm.~2.1]{fiacco76}, to the subproblem.
We get that there exists $\delta > 0$ and continuously differentiable functions
$\x^+$, $\bmu^+$
defined on a neighborhood $N_{\delta}(\y^*,\blam^*)$ such that
$(\x^+(\y^k,\blam^k),\bmu^+(\y^k,\blam^k))$ 
is a local minimizer of the subproblem for 
$(\y^k,\blam^k) \in N_{\delta}(\y^*,\blam^*)$, 
and furthermore is the unique (only) local minimizer in a neighborhood of $(\x^*,\bmu^*)$. 
As well, 
$(\x^+(\y^*,\blam^*),\bmu^+(\y^*,\blam^*)) = (\x^*,\bmu^*)$.
Then Assumption~\ref{assm:nearestSolution} implies that if
$(\y^k,\blam^k) \in N_{\delta}(\y^*,\blam^*)$, 
the closest local minimizer
$(\x^{k+1},\bmu^{k+1})$
must coincide with
$(\x^+(\y^k,\blam^k),\bmu^+(\y^k,\blam^k))$
(because it is unique)
and so the result follows.
\end{proof}

Noting the affine dependence of $\y^{k+1}$ and $\blam^{k+1}$ on $\x^{k+1}$, we build on Lemma~\ref{lem:continuousXMU} to establish, essentially, that the mapping
$(\y^k,\blam^k) \mapsto (\y^{k+1},\blam^{k+1})$ 
is continuous and has a fixed point at $(\y^*,\blam^*)$.

\begin{lemma}
\label{lem:continuousYLambda}
Let Assumptions~\ref{assm:nearestSolution} and \ref{assm:SPRegularity} hold and assume that $\mbf{B}$ has full column rank.
For all $\epsilon_2 > 0$, there exists $\delta_2 > 0$ such that,
if 
$\norm{(\y^k,\blam^k) - (\y^*,\blam^*)} < \delta_2$
then
$\norm{(\y^{k+1},\blam^{k+1}) - (\y^*,\blam^*)} < \epsilon_2$.
\end{lemma}
\begin{proof}
If $\mbf{B}$ has full column rank, then we can use formula~\eqref{eq:y_update} for $\y^{k+1}$.
Now, to make the following arguments as precise as possible, define
\[
\begin{aligned}
&\y^+   : (\x) 			\mapsto (\mbf{B}\tr\mbf{B})\inv \mbf{B}\tr (\b - \mbf{A}\x), \\
&\blam^+: (\x,\y,\blam) 	\mapsto \blam + \rho(\mbf{A}\x + \mbf{B}\y - \b).
\end{aligned}
\]
It is clear that $(\y^+,\blam^+)$ is continuous.
We also have
\[
\begin{aligned}
&\y^+(\x^{k+1}) = \y^{k+1}, \quad	&&\blam^+(\x^{k+1},\y^{k+1},\blam^k) = \blam^{k+1}, \\
&\y^+(\x^*) = \y^*,				&&\blam^+(\x^*,\y^*,\blam^*) = \blam^*.
\end{aligned}
\]
where the bottom relations follow from 
$\mbf{A}\x^* + \mbf{B}\y^* = \b$.

Using Lemma~\ref{lem:continuousXMU}, 
there is a continuous function $\x^+$ so that
$\x^+(\y^*,\blam^*) = \x^*$ 
and
$\x^+(\y^k,\blam^k) = \x^{k+1}$.
Thus we have
\[\begin{aligned}
&\y^+(\x^+(\y^k,\blam^k)) = \y^{k+1}, \\
&\blam^+(\x^+(\y^k,\blam^k), \y^+(\x^+(\y^k,\blam^k)), \blam^k)  = \blam^{k+1}.
\end{aligned}\]
Noting that the composition of continuous functions is continuous, we see the ``continuous dependence'' of $(\y^{k+1},\blam^{k+1})$ on $(\y^k,\blam^k)$ 
(and the fixed point at $(\y^*,\blam^*)$) 
and so the result follows.
\end{proof}

To simplify the proof of the main theorem, we combine Lemmata~\ref{lem:continuousXMU} and \ref{lem:continuousYLambda}.

\begin{lemma}
\label{lem:continuousEverything}
Let Assumptions~\ref{assm:nearestSolution} and \ref{assm:SPRegularity} hold and assume that $\mbf{B}$ has full column rank.
For all $\epsilon_3 > 0$, there exists $\delta_3 > 0$ such that,
if 
$\norm{(\y^k,\blam^k) - (\y^*,\blam^*)} < \delta_3$
then
\[
\norm{(\x^{k+1},\y^{k+1},\bmu^{k+1},\blam^{k+1}) - (\x^*,\y^*,\bmu^*,\blam^*)} < \epsilon_3.
\]
\end{lemma}
\begin{proof}
Follows from Lemmata~\ref{lem:continuousXMU} and \ref{lem:continuousYLambda} and equivalence of norms on finite dimensional spaces.
\end{proof}

\subsection{Main result}

We now state the main convergence result:
if the penalty parameter is sufficiently large, and if for some iteration, $\y^k$ and $\blam^k$ are sufficiently close to the optimal values, then we have convergence.
Compared with Proposition~\ref{prop:almostConvergence}, the following result asserts that there is a neighborhood of the solution which captures the iterates.
\begin{theorem}
\label{thm:convergence}
Let Assumptions~\ref{assm:KKTiterates}, \ref{assm:OverallRegularity}, and \ref{assm:nearestSolution} hold.
There exist positive constants $\epsilon'$ and $\rho^*$ such that,
if 
$\rho > \rho^*$ and
$\norm{(\y^{K},\blam^{K}) - (\y^*,\blam^*)} < \epsilon'$
for some $K$,
then 
$\seq[k]{(\x^{k+1},\y^{k+1},\bmu^{k+1},\blam^{k+1})}$ 
converges to $(\x^*,\y^*,\bmu^*,\blam^*)$.
\end{theorem}
\begin{proof}
Our goal is to show that the conditions of Proposition~\ref{prop:almostConvergence} hold and apply that result;
we need to show that for $\epsilon$, $\rho^*$ guaranteed to exist by that result, that we have for all sufficiently large $k$, 
$\norm{(\x^{k+1},\y^{k+1},\bmu^{k+1},\blam^{k+1}) - (\x^*,\y^*,\bmu^*,\blam^*)} < \epsilon$
and
$\rho$ is greater than $\rho^*$.
This last condition holds by assumption.
By the analysis in Appendix~\ref{app:RegularityRelation}, Assumption~\ref{assm:OverallRegularity} implies that Assumption~\ref{assm:SPRegularity} holds for sufficiently large $\rho$
(and that $\mbf{B}$ has full column rank);
thus without loss of generality we can assume that $\rho$ is large enough that  Assumption~\ref{assm:SPRegularity} holds
(effectively, redefining $\rho^*$ if necessary).
So, we can apply Lemma~\ref{lem:continuousEverything} to see that there exists $\delta_3$ such that
$\norm{(\y^{k},\blam^{k}) - (\y^*,\blam^*)} < \delta_3$
implies
$\norm{(\x^{k+1},\y^{k+1},\bmu^{k+1},\blam^{k+1}) - (\x^*,\y^*,\bmu^*,\blam^*)} < \epsilon$.

Consider the expression appearing in the definition of the Lyapunov function $V^k$ in Equation~\eqref{eq:lyapunov}.
We have that 
\begin{equation}
\label{eq:rhoNorm}
\norm{(\y,\blam)}_{\rho} \equiv 
\left( \rho\norm{\mbf{B}\y}^2 + \frac{1}{\rho}\norm{\blam}^2 \right)^{\frac{1}{2}}
\end{equation}
is in effect a scaled 2-norm (since $\mbf{B}$ has full column rank under Assumption~\ref{assm:OverallRegularity} and $\rho >0$).
Using the equivalence of norms of finite dimensional spaces, there exist positive constants $C_1$, $C_2$, with $C_1 \le C_2$, so that
\[
C_1 \norm{(\y,\blam)}_{\rho} \le \norm{(\y,\blam)} \le C_2 \norm{(\y,\blam)}_{\rho}.
\]
Consequently, if 
$\norm{(\y,\blam)}_{\rho} < \frac{1}{C_2}\delta_3$, then
$\norm{(\y,\blam)} < \delta_3$.

We proceed with an induction argument.
Assume that for some $k$, we have
$\norm{(\y^{k},\blam^{k}) - (\y^*,\blam^*)} < \delta_3$ and
$\norm{(\y^{k},\blam^{k}) - (\y^*,\blam^*)}_{\rho} < \frac{1}{C_2}\delta_3$.

By the preceding arguments, we have
$\norm{(\x^{k+1},\y^{k+1},\bmu^{k+1},\blam^{k+1}) - (\x^*,\y^*,\bmu^*,\blam^*)} < \epsilon$.
We can then apply Proposition~\ref{prop:LyapunovDecreaseBound} which implies that $V^{k+1}$ must be less than or equal to $V^k$.
Using the definition of the norm in \eqref{eq:rhoNorm}, this means
\begin{equation}
\label{ineq:normDecreases}
\norm{(\y^{k+1},\blam^{k+1}) - (\y^*,\blam^*)}_{\rho}^2
\le 
\norm{(\y^{k},\blam^{k}) - (\y^*,\blam^*)}_{\rho}^2 
< (\frac{1}{C_2}\delta_3)^2.
\end{equation}
Then 
$\norm{(\y^{k+1},\blam^{k+1}) - (\y^*,\blam^*)}_{\rho} < \frac{1}{C_2}\delta_3$, 
which then implies 
$\norm{(\y^{k+1},\blam^{k+1}) - (\y^*,\blam^*)} < \delta_3$.

Thus, we have established that the induction hypothesis holds for $k+1$ and proved the induction step;
it remains to show that we have an induction basis.
This follows from the conditions of the theorem for $k = K$, noting that 
$\norm{(\y,\blam)} <  \frac{C_1}{C_2}\delta_3$
implies
$\norm{(\y,\blam)}_{\rho} < \frac{1}{C_2}\delta_3$,
and then taking
$\epsilon' = \frac{C_1}{C_2}\delta_3 \le \delta_3$.
Thus we have established the conditions of Proposition~\ref{prop:almostConvergence} and convergence follows.
\end{proof}

A sublinear rate of convergence falls out naturally, using the definitions from \cite[\S1.2]{bertsekas_lmm}.

\begin{corollary}
\label{cor:convergenceRate}
Let Assumptions~\ref{assm:KKTiterates}, \ref{assm:OverallRegularity}, and \ref{assm:nearestSolution} hold.
There exist positive constants $\epsilon'$ and $\rho^*$ such that,
if 
$\rho > \rho^*$ and
$\norm{(\y^{K},\blam^{K}) - (\y^*,\blam^*)} < \epsilon'$
for some $K$,
then 
$\seq[k]{\norm{(\y^k,\blam^k) - (\y^*,\blam^*)}_{\rho}}$ 
converges to zero sublinearly or in finite iterations.
\end{corollary}
\begin{proof}
We can revisit Inequality~\eqref{ineq:normDecreases} in the proof of Theorem~\ref{thm:convergence}
\begin{equation}
\notag
\norm{(\y^{k+1},\blam^{k+1}) - (\y^*,\blam^*)}_{\rho}^2
\le
\norm{(\y^{k},\blam^{k}) - (\y^*,\blam^*)}_{\rho}^2.
\end{equation}
If 
$\seq[k]{\norm{(\y^{k},\blam^{k}) - (\y^*,\blam^*)}_{\rho}}$
does not not converge to zero in finite iterations, we can assume that it is nonzero for all $k$, and thus we obtain 
\begin{equation}
\notag
\frac{\norm{(\y^{k+1},\blam^{k+1}) - (\y^*,\blam^*)}_{\rho}}{\norm{(\y^{k},\blam^{k}) - (\y^*,\blam^*)}_{\rho}} \le 1
\end{equation}
for all $k$.
Consequently, the result follows from \cite[Prop.~1.1]{bertsekas_lmm}, or this can be seen directly as the definition of Q (``quotient'')-sublinear convergence.
\end{proof}


\section{Discussion}
\label{sec:discussion}

In this section we will try to provide some further context for the results and assumptions behind them.
The following section will analyze a few examples to illustrate other points further.

\subsection{Compared with other work}
The modifications to the convergence result, compared to the analysis in \cite{boydEA11}, are in part inspired by analysis of the method of multipliers.
See, specifically, \cite[\S2.2.3]{bertsekas_lmm}, which uses properties of a primal functional and penalized primal functional as a key analytical tool.
As well, \cite[Prop.~2.14]{bertsekas_lmm}, which deals with inexact minimization of the augmented Lagrangian in the method of multipliers setting, bears similarity to Lemma~\ref{lem:primalProperties}.
Specifically, from Lemma~\ref{lem:step2} (or using Equations~\eqref{eq:toGetKKT} and \eqref{eq:y_stationarity}), we can identify the iterates $(\x^{k+1},\y^{k+1})$ as inexact solutions of a method of multipliers subproblem, and the dual residual equals the ``inexactness.''
The analysis in \cite{bertsekas_lmm} following Prop.~2.14 prescribes a method in which the level of inexactness is defined to go zero.
Meanwhile, the present analysis does not and cannot directly force the dual residuals to zero; 
we must rely on the arguments involving the Lyapunov function to show that the dual residuals converge to zero.

The recent work in \cite{wangEA15} presents very general conditions under which ADMM converges, and it is plausible that with some effort those results could be used to prove some version of the present results.
The main challenge in the present setting is meeting the assumption of ``Lipschitz sub-minimization paths'' from \cite{wangEA15}.
Effectively, this assumption plays a similar role to the sensitivity analysis employed here (in Lemma~\ref{lem:primalProperties} for example), but at a more global level.
In our notation, the assumption includes the requirement that 
$\u \mapsto \arg\min_{\x} \set{f(\x) + \iota_c(\x) : \mbf{A}\x = \u }$
is singleton-valued and Lipschitz continuous, for all $\u$ in the image/range of $\mbf{A}$, where $\iota_c$ is the indicator function of the constraints $\set{ \x : \c(\x) = \0}$
(that is, $\iota_c(\x) = 0$ if $\c(\x) = \0$, and $+\infty$ otherwise).
Using similar arguments as in Lemma~\ref{lem:primalProperties}, this mapping is indeed singleton-valued and Lipschitz continuous \emph{locally} around $\u^* = \b - \mbf{B}\y^*$ under Assumption~\ref{assm:OverallRegularity}.
However, define
$f : (x_1,x_2) \mapsto x_2$,
$c : (x_1,x_2) \mapsto x_1^2 - x_2$, and
$\mbf{A} = [0 \; 1]$;
the parametric problem 
$\min_{\x} \set{ f(\x) : c(\x) = 0, \mbf{A}\x = u}$
becomes
$\min_{\x} \set{ x_2 : x_1^2 = x_2, x_2 = u }$.
It's clear that for $u > 0$, the problem has isolated local minima $(\sqrt{u},u)$ and $(-\sqrt{u},u)$;
these are also global minima (so while local minima are unique, global minima are not),
and the problem is infeasible for $u < 0$.
When considering the reformulation with the indicator function
($\min_{\x} \set{f(\x) + \iota_c(\x) : \mbf{A}\x = \u }$)
the solution set is $\set{ (x_1, x_2) : x_2 = u}$ for $u < 0$, since the objective function equals $+\infty$ for all feasible $\x$.
Clearly this is not a singleton, and in the present approach we recover this required regularity by restricting ourselves to a local analysis with local minimizers.
However, if the assumption of Lipschitz sub-minimization paths is satisfied, stronger conclusions are possible; 
namely, \emph{global} convergence results can be shown as in \cite{wangEA15}.

The main local convergence result in Theorem~\ref{thm:convergence} is similar to the recent work in \cite{chatzipanagiotisEA17}.
The assumptions required for their convergence result are similar to those required here;
in particular, both results show convergence to a local optimal solution which is assumed to satisfy second order sufficient conditions.
Differences include the fact that nonconvex constraints are allowed in the present work.
Meanwhile, the notion of a step-size is included in the analysis of \cite{chatzipanagiotisEA17}, and this seems to contribute to an observed improvement in robustness and convergence rates;
see \cite[{\S}IV]{chatzipanagiotisEA17}.

\subsection{Local versus global solution}
Another difference between the present work and \cite{chatzipanagiotisEA17} is the presence of Assumption~\ref{assm:nearestSolution}, that the closest local minimizer to $\x^*$ is found at each iteration.
This difference seems to be due to whether a global or local minimizer of \eqref{ADMMSP} is found at each iteration.
Indeed, if we assume that $(\x^*,\y^*)$ is a \emph{global} minimum of Problem~\eqref{problem} satisfying Assumption~\ref{assm:OverallRegularity}, and that a \emph{global} minimizer of \eqref{ADMMSP} is found at each iteration, it might be possible to use the sensitivity result from \cite[Thm.~4.1]{shapiroEA04} 
(which the authors of \cite{chatzipanagiotisEA17} seem to cite in the proof of their Lemma~5)
to modify the analysis, and show convergence to a global minimizer of Problem~\eqref{problem}.

However, if we do not wish to solve the subproblem globally, this assumption is unavoidable;
consider
\[
	\min_{x} \set{ f(x) + \phi(x,\blam,\y,\rho) : x^2 - 1 = 0 }.
\]
The feasible set of this problem is $\set{+1, -1}$, and no matter the definition of $f$ or $\phi$, this subproblem will always have two local minima, and some assumption must be made to resolve which is found.
As mentioned, something like Assumption~\ref{assm:nearestSolution} is made in the analysis of the classic method of multipliers, and in practice one would likely supply $\x^k$ as the initial guess when solving \eqref{ADMMSP} to obtain $\x^{k+1}$.
Then if $\x^k$ is close to $\x^*$, a well-behaved local solver should produce a solution $\x^{k+1}$ which is close to $\x^*$ as well.

\subsection{Convergence rates and other considerations}
The convergence rate from Corollary~\ref{cor:convergenceRate} is a little disappointing for a few reasons.
For one, under Assumption~\ref{assm:OverallRegularity} other optimization methods achieve, for instance, at least superlinear convergence
(see \cite[\S19.8]{nocedal_wright} for a high level discussion in the context of interior-point methods).
Further, numerical studies, like in \cite{rodriguezEA18}, show that the primal and dual errors, at least, seem to display linear convergence rates.
We also note that the numerical studies in \cite{rodriguezEA18} show that the Lyapunov function may decrease non-monotonically in some cases, which is at odds with Proposition~\ref{prop:LyapunovDecreaseBound} which indicates that the decrease must be monotonic in the setting of Theorem~\ref{thm:convergence}.
This hints that more general results are possible.

As the convergence rate result Corollary~\ref{cor:convergenceRate} makes explicit, the natural norm appearing in this analysis is $\norm{\cdot}_{\rho}$.
As mentioned before, this norm is inspired by the form of the Lyapunov function from \eqref{eq:lyapunov}, and in fact the convergence rate result directly implies that $\seq[k]{\sqrt{V^k}}$ converges to zero sublinearly.
From a geometric perspective, the norm is troublesome, especially as $\rho$ increases.
This is because, for $\rho > 1$, the neighborhood around $(\y^*,\blam^*)$ that we must ``hit'' for convergence to occur is an ellipse that is elongated in the $\blam$ dimension and shortened in the $\y$ dimension.
This indicates that choosing a very large value of $\rho$ may make it more difficult to choose an appropriate initial guess $\y^0$.

\section{Illustrative examples}
\label{sec:examples}

We will not consider particularly extensive numerical studies here;
the recent work in \cite{rodriguezEA18} presents excellent numerical studies of the performance of ADMM as well as the method of multipliers and other variants in the nonconvex setting.
Instead, we will analyze a few examples to illustrate a few points.

\subsection{Convergence of a small example}
This example comes from \cite[Example~2.1]{houskaEA16}, which claims that a particular variant of ADMM does \emph{not} converge for the example.
In contrast, we apply Algorithm~\ref{alg:ADMM} and verify that convergence does indeed happen.
This highlights why a detailed analysis of specific forms of ADMM, like this work, are important for understanding the situations and assumptions under which ADMM converges.

Consider 
$
\min_{x_1,x_2}\set{ x_1 x_2 : x_1 = x_2},
$
which in the formulation considered in this work is
\begin{align*}
\min_{x_1,x_2,y}\;& x_1 x_2\\
\st
&x_1 - y = 0,\\
&x_2 - y = 0.
\end{align*}
It is clear the solution is $x_1 = x_2 = y =0$.
Let
$\mbf{H} = \begin{bmatrix}0 & 1\\1 & 0\end{bmatrix}$,
$\mbf{A} = \mbf{I}$, $\mbf{B} = [-1, -1]\tr$, $\b = \0$.
The subproblem~\eqref{ADMMSP} 
in this case has first order stationary conditions 
\[
\mbf{H}\x + \blam^k + \rho\mbf{A}\tr(\mbf{A}\x + \mbf{B}y^k) = \0 
\]
which imply the update formula for $\x$:
\[
(\mbf{H} + \rho\mbf{I})\x^{k+1} = -(\blam^k + \rho\mbf{B}y^k).
\]
Let $\mbf{R} = (\mbf{H} + \rho\mbf{I})\inv$ which is indeed invertible for $\rho > 1$.
With this formula giving the update of $\x$, we can write the updates for $y$ and $\blam$.
We have
\[
y^{k+1} 
	= (\mbf{B}\tr\mbf{B})\inv\mbf{B}\tr(-\x^{k+1}) 
	= \frac{1}{2}\mbf{B}\tr \mbf{R}(\blam^k + \rho\mbf{B}y^k)
\]
where we note that $\mbf{B}\tr\mbf{B} = 2$,
and
\[
\begin{aligned}
\blam^{k+1} 
	&= \blam^k + \rho(\x^{k+1} + \mbf{B}y^{k+1}) \\
	&= (\mbf{I} - \rho\mbf{R} + \frac{\rho}{2}\mbf{B}\mbf{B}\tr \mbf{R})\blam^k
   + (\frac{\rho^2}{2}\mbf{B}\mbf{B}\tr \mbf{R}\mbf{B} - \rho^2\mbf{R}\mbf{B})y^k.
\end{aligned}
\]
Writing out these updates for one big system in $(\blam^{k+1},y^k)$ we get
\[
\begin{bmatrix}
\blam^{k+1} \\ y^{k+1}
\end{bmatrix} 
=
\begin{bmatrix}
\mbf{I} + \rho(\frac{1}{2}\mbf{B}\mbf{B}\tr - \mbf{I})\mbf{R} & 
	\rho^2(\frac{1}{2}\mbf{B}\mbf{B}\tr - \mbf{I})\mbf{R}\mbf{B} \\
\frac{1}{2}\mbf{B}\tr\mbf{R} &
	\frac{\rho}{2}\mbf{B}\tr\mbf{R} \mbf{B}
\end{bmatrix}
\begin{bmatrix}
\blam^{k} \\ y^{k}
\end{bmatrix}.
\]
Call the matrix defining the iteration above $\mbf{M}$. 
Noting that 
\[
\mbf{R} = 
\begin{bmatrix}
\frac{-\rho}{1-\rho^2} & \frac{1}{1-\rho^2} \\
\frac{1}{1-\rho^2} & \frac{-\rho}{1-\rho^2}
\end{bmatrix}
\]
we can write $\mbf{M}$ explicitly in terms of $\rho$ to get
\[
\mbf{M} = 
\begin{bmatrix}
\frac{1 - \sfrac{\rho}{2}}{1-\rho} & \frac{-\sfrac{\rho}{2}}{1-\rho} & 0 \\
\frac{-\sfrac{\rho}{2}}{1-\rho} & \frac{1 - \sfrac{\rho}{2}}{1-\rho} & 0 \\
\frac{-\sfrac{1}{2}}{1+\rho} & \frac{-\sfrac{1}{2}}{1+\rho} & \frac{\rho}{1+\rho}
\end{bmatrix}.
\]
It is easy to verify that the eigenvalues of $\mbf{M}$ are 
$1$, $\sfrac{\rho}{(1 + \rho)}$, and $\sfrac{1}{(1 - \rho)}$,
with corresponding eigenvectors
$({-}1,{-}1,1)$, $(0,0,1)$, and $(1,{-}1,0)$, respectively.
We note that for $\rho > 2$, the absolute values of $\sfrac{\rho}{(1 + \rho)}$ and $\sfrac{1}{(1 - \rho)}$ are both strictly less than $1$.
In this case, the iterated linear system converges to some vector in the eigenspace of the eigenvalue that equals $1$, i.e. the span of $({-}1,{-}1,1)$
(see for instance \cite[\S5.3]{strang}). 
Let this solution be $(\blam^*,y^*)$.
Recall that Algorithm~\ref{alg:ADMM} enforces $\mbf{B}\tr\blam^k = \0$ for all $k$.
(See the discussion following Equation~\eqref{eq:y_stationarity}.
Algorithm~\ref{alg:ADMM} requires $\mbf{B}\tr \blam^0 = \0$, and for this particular example this becomes $\lambda_1^{k} + \lambda_2^{k} = 0$;
from the iteration defined by the matrix $\mbf{M}$ above, we also have 
$\lambda_1^{k+1} + \lambda_2^{k+1} = \lambda_1^{k} + \lambda_2^{k}$.)
Consequently, we see that $\mbf{B}\tr\blam^* = \0$ as well since the nullspace of $\mbf{B}\tr$ is closed.
Clearly the only point $(\lambda_1,\lambda_2,y)$ in the span of $({-}1,{-}1,1)$ that also satisfies $\lambda_1 = -\lambda_2$ is $(0,0,0)$, which indeed corresponds to the solution of the optimization problem.


In this analysis, we came across two critical values of $\rho$ in order for convergence to occur.
It is tempting to try to match these thresholds with those in the theory coming from Lemmata~\ref{lem:primalProperties} and \ref{lem:step1}.
In \S\ref{sec:ex:critical_rho}, we explore this further.

\subsection{Different step lengths}
In some versions of ADMM, like in \cite{chatzipanagiotisEA17}, the update for $\blam$ includes an extra parameter to control the step length.
We modify the previous example to explore what affect that might have.

The formulas for $\x^{k+1}$ and $y^{k+1}$ remain the same:
\[
\begin{aligned}
\x^{k+1} 
	&= -\mbf{R}(\blam^k + \rho\mbf{B}y^k) \\
y^{k+1} 
	&= \frac{1}{2}\mbf{B}\tr \mbf{R}(\blam^k + \rho\mbf{B}y^k).
\end{aligned}
\]
We add a step length parameter $\tau$ to the update for $\blam^{k+1}$:
\[
\begin{aligned}
\blam^{k+1} 
	&= \blam^k + \tau\rho(\x^{k+1} + \mbf{B}y^{k+1}) \\
	&= (\mbf{I} - \tau\rho\mbf{R} + \frac{\tau\rho}{2}\mbf{B}\mbf{B}\tr \mbf{R})\blam^k
   + (\frac{\tau\rho^2}{2}\mbf{B}\mbf{B}\tr \mbf{R}\mbf{B} - \tau\rho^2\mbf{R}\mbf{B})y^k.
\end{aligned}
\]
The modified iteration for $(\blam^{k+1},y^k)$ is
\[
\begin{bmatrix}
\blam^{k+1} \\ y^{k+1}
\end{bmatrix} 
=
\begin{bmatrix}
\mbf{I} + \tau\rho(\frac{1}{2}\mbf{B}\mbf{B}\tr - \mbf{I})\mbf{R} & 
	\tau\rho^2(\frac{1}{2}\mbf{B}\mbf{B}\tr - \mbf{I})\mbf{R}\mbf{B} \\
\frac{1}{2}\mbf{B}\tr\mbf{R} &
	\frac{\rho}{2}\mbf{B}\tr\mbf{R} \mbf{B}
\end{bmatrix}
\begin{bmatrix}
\blam^{k} \\ y^{k}
\end{bmatrix}
\]
or
\[
\begin{bmatrix}
\lambda_1^{k+1} \\ \lambda_2^{k+1} \\ y^{k+1}
\end{bmatrix} 
=
\begin{bmatrix}
\frac{1 - \rho(1 - \sfrac{\tau}{2})}{1-\rho} & \frac{-\sfrac{\tau\rho}{2}}{1-\rho} & 0 \\
\frac{-\sfrac{\tau\rho}{2}}{1-\rho}          & \frac{1 - \rho(\sfrac{1 - \tau}{2})}{1-\rho} & 0 \\
\frac{-\sfrac{1}{2}}{1+\rho}                 & \frac{-\sfrac{1}{2}}{1+\rho} & \frac{\rho}{1+\rho}
\end{bmatrix}
\begin{bmatrix}
\lambda_1^{k} \\ \lambda_2^{k} \\ y^{k}
\end{bmatrix}.
\]

Repeating the analysis from before, we see that the matrix defining the iteration has an eigenvalue equal to $1$, with corresponding eigenvector $(-1, -1, 1)$, and two other eigenvalues 
$\sfrac{\rho}{(1 + \rho)}$
and
$\sfrac{(1 - \rho(1 - \tau))}{(1 - \rho)}$.
For any positive $\rho$, 
$\sfrac{\rho}{(1 + \rho)} \in (0,1)$, 
while for any $\rho > 1$ and $\tau$ such that $0 < \tau < \min\set{1, \sfrac{2(\rho-1)}{\rho}}$, we have
$\abs{\sfrac{(1 - \rho(1 - \tau))}{(1 - \rho)}} < 1$.
Once again, if we impose $\mbf{B}\tr \blam^0 = \0$, then again we can show by induction that $\mbf{B}\tr \blam^k = \0$ for all $k$.
Thus, as before, the iteration converges to the solution $(\lambda_1^*, \lambda_2^*, y) = (0,0,0)$.
Critically though, using a step length $\tau$ permits us to keep $\rho$ smaller than before
(where we had required $\rho > 2$).

We could try to modify the analysis to permit a step length;
however, it is unclear what practical impact it might have.
For one, the critical values of $\rho$ from Lemmata~\ref{lem:primalProperties} and \ref{lem:step1} are essentially independent of the updates used in the algorithm.
Further, a step length is likely to impact the convergence \emph{rate}, and the present analysis only makes very weak statements about convergence rates.
A different analytical approach is needed to take full advantage of a step length.

\subsection{Critical value of $\rho$}
\label{sec:ex:critical_rho}

Next we study an example to assess the critical value of $\rho$ predicted by the theory.
Ultimately, this reiterates that this threshold value is not something that one can compute in practice; 
still, this example is interesting as it demonstrates that due to constraints, the critical value of $\rho$ can be zero, even in the nonconvex setting.

Consider 
$
\min_{\x} \set{\frac{1}{2}x_1^2 + \frac{1}{2}x_2^2 - x_3^2 : x_1 = x_2, x_1^2 + x_3^2 = 1 }
$,
which is reformulated as 
\[\begin{aligned}
\min_{\x,y}\; &\frac{1}{2}x_1^2 + \frac{1}{2}x_2^2 - x_3^2\\
\st
&x_1 - y = 0,\\
&x_2 - y = 0,\\
&x_1^2 + x_3^2 = 1.
\end{aligned}\]
Because of the constraint $x_1 = x_2$, this is equivalent to a two-dimensional problem and it is easy to see that the solution is $(x_1^*, x_2^*, x_3^*, y^*) = (0, 0, 1, 0)$.
The optimal multipliers are $(\lambda_1^*, \lambda_2^*, \mu^*) = (0,0,1)$.
It is fairly simple to verify that this solution satisfies Assumption~\ref{assm:OverallRegularity}.
For instance, the Jacobian of the active constraints is 
\[\mbf{C} = 
\begin{bmatrix}
1 & 0 & 0 & -1 \\
0 & 1 & 0 & -1 \\
0 & 0 & 2 &  0
\end{bmatrix}
\]
which has linearly independent rows;
further, it has nullspace equal to $\text{span}\set{(1, 1, 0, 1)}$.
The partial Hessian of the Lagrangian is
\[
\mbf{H}_{xx}(\x^*, \mu^*) = 
\begin{bmatrix}
3 & 0 & 0 \\
0 & 1 & 0 \\
0 & 0 & 0
\end{bmatrix}
\]
and we can verify that 
$\begin{bmatrix} \mbf{H}_{xx}(\x^*, \mu^*) & \0 \\ \0 & 0 \end{bmatrix}$
is positive definite on the nullspace of $\mbf{C}$.

In order to determine the critical value of $\rho$ required for convergence, we first look at Assumption~\ref{assm:SPRegularity}.
We require 
\[
\mbf{H}_{xx}(\x^*, \mu^*) + \rho\mbf{A}\tr\mbf{A}
=
\begin{bmatrix}
3 & 0 & 0 \\
0 & 1 & 0 \\
0 & 0 & 0
\end{bmatrix}
+
\rho
\begin{bmatrix}
1 & 0 & 0 \\
0 & 1 & 0 \\
0 & 0 & 0
\end{bmatrix}
\]
to be positive definite on $\text{span}\set{(1,0,0), (0,1,0)}$, the nullspace of 
$\begin{bmatrix} 0 & 0 & 2 \end{bmatrix}$.
We see this holds for $\rho > 0$.

Next, we look at Lemma~\ref{lem:step1}.
We need to establish that $(\x^*, y^*)$ is a local minimizer of the augmented Lagrangian evaluated at the optimal multipliers:
\[
L_{\rho} : (\x, y) \mapsto 
	\frac{1}{2}x_1^2 + \frac{1}{2}x_2^2 - x_3^2 + 
	\frac{\rho}{8} ( (x_1 - y)^2 + (x_2 - y)^2) + 
	(x_1^2 + x_3^2 - 1) + \frac{\rho}{2} (x_1^2 + x_3^2 - 1)^2
\]
Note that we have used $\blam^* = \0$ and $\mu^* = 1$.
To determine the value of $\rho$ required to make $(\x^*, y^*)$ a local minimizer, we evaluate the Hessian;
we get
\[ \grad^2 L_{\rho}(\x^*, y^*) = 
\begin{bmatrix}
3 + \sfrac{\rho}{4} & 0 & 0 & -\sfrac{\rho}{4} \\
0 & 1 + \sfrac{\rho}{4} & 0 & -\sfrac{\rho}{4} \\
0 & 0 & 4\rho & 0 \\
-\sfrac{\rho}{4} & -\sfrac{\rho}{4} & 0 & \sfrac{\rho}{2}
\end{bmatrix}
\]
The top left $3 \times 3$ submatrix is positive definite for any $\rho > 0$.
Using Schur complements we can verify that $\grad^2 L_{\rho}(\x^*, y^*)$ is positive definite for any $\rho > 0$
(the Schur complement of the top left $3 \times 3$ submatrix is the $1 \times 1$ matrix with entry 
$\sfrac{(128\rho^2 + 768 \rho)}{(32\rho^2 + 512 \rho + 1536)}$;
evidently this matrix is positive definite for any $\rho > 0$).
Not surprisingly, this is the same threshold value required above by Assumption~\ref{assm:SPRegularity}.

The final step is to determine the primal functional defined in Lemma~\ref{lem:primalProperties}:
\[
	p(\r,\q) = \min_{\x,y}\set{\frac{1}{2}x_1^2 + \frac{1}{2}x_2^2 - x_3^2 - \r\tr\x : 
		x_1^2 + x_3^2 = 1, x_1 - y = q_1, x_2 - y = q_2 }
\]
Specifically, we need to determine the value of $\rho$ so that 
$\grad_{qq}^2 p(\0,\0) + \rho \mbf{I} \symgt \0$.
Simplifying the problem to a one-dimensional problem in $y$ by making the substitutions
$x_1 = y + q_1$, 
$x_2 = y + q_2$,
$x_3^2 = 1 - (y + q_1)^2$,
we get 
$p(\0, \q) = \min_y \set{2y^2 + (3q_1 + q_2)y + \frac{3}{2}q_1^2 + \frac{1}{2}q_2^2 - 1 }$.
We can solve this for the optimal $y$, and obtain explicitly
$p(\0, \q) = \frac{3}{8}(q_1 - q_2)^2 - 1$.
Thus
$\grad_{qq}^2 p(\0,\0) = \sfrac{3}{4} \begin{bmatrix} 1 & -1 \\ -1 & 1 \end{bmatrix}$
and we see that $\grad_{qq}^2 p(\0,\0) + \rho \mbf{I} \symgt \0$ for any $\rho > 0$.

Consequently, the theory predicts that $\rho > 0$ should suffice for convergence, assuming $y^0$ and $\blam^0$ are sufficiently close to the optimal values $y^*$ and $\blam^*$.
We test this numerically.
Setting $\rho = 10^{-1}$, we sample $y^0$ and $\hat{\lambda}$ uniformly from $[-2, 2]$;
subsequently we set $\blam^0 = (\hat{\lambda}, -\hat{\lambda})$
(in order to enforce $\mbf{B}\tr \blam^0 = \0$ as required by Algorithm~\ref{alg:ADMM}).
Setting an iteration limit of 500, we apply Algorithm~\ref{alg:ADMM}.
Repeating this 100 times for the different random starting values, we see that the worst error over these 100 trials is less than $3 \times 10^{-14}$
(specifically, for each trial, $\max\set{\abs{y^{500} - y^*}, \norm{\blam^{500} - \blam^*}} < 3\times 10^{-14}$).

To calculate that we required $\rho > 0$, we needed to know the solution beforehand, which is why this critical value is not something we can know in practice.
However, knowing that a critical value \emph{exists} is still useful;
meanwhile this shortcoming of the theory highlights the need for future work.
\section{Conclusions}
\label{sec:conclusion}
This work has investigated the theoretical performance of the alternating-direction method of multipliers as it applies to nonconvex optimization problems.
These theoretical contributions culminate in a local convergence result.
This result helps explain the empirical performance of ADMM often observed on nonconvex problems from engineering applications. 
The result permits nonconvex constraints in the subproblems and does not specify how these constraints must be handled in the solution of the subproblems, which is important to practical implementations that take advantage of state-of-the-art local optimization software.

\section*{Financial and ethical disclosures}
This work was not funded by any grants. 
The author declares that he has no conflict of interest.

\subsection*{Acknowledgments}
The author would like to thank his colleagues Shivakumar Kameswaran, Thomas Badgwell, and Francisco Trespalacios for fruitful discussions in developing this work.
\appendix

\section{Regularity of overall problem and subproblems}
\label{app:RegularityRelation}

This section establishes the claim that Assumption~\ref{assm:OverallRegularity} implies Assumption~\ref{assm:SPRegularity};
that is, that a solution of the main problem~\eqref{problem} which satisfies conditions including the second-order sufficient conditions implies that the subproblems also have solutions satisfying similar conditions.
This is established through the following lemmata.
This next result is a modification of \cite[Lemma~3.2.1]{bertsekas_nlp}.

\begin{lemma}
\label{lem:PD}
Let $\mbf{H} \in \mbb{R}^{n \times n}$ be a symmetric matrix and let $\mbf{C} \in \mbb{R}^{p \times n}$ and $\mbf{D} \in \mbb{R}^{p' \times n}$.
Assume that $\mbf{H}$ is positive definite on the nullspace of $\begin{bmatrix} \mbf{C} \\ \mbf{D} \end{bmatrix}$:
$\z \tr \mbf{H} \z > 0$ for all $\z \neq \0$ with $\mbf{C}\z = \0$ and $\mbf{D}\z = \0$.
Then there exists $\rho^*$ such that for all $\rho > \rho^*$,
\[
\z\tr (\mbf{H} + \rho \mbf{D}\tr \mbf{D}) \z > 0
\]
for all $\z \neq \0$ with $\mbf{C}\z = \0$.
\end{lemma}
\begin{proof}
Assume the contrary.
Then for all $k \in \mbb{N}$, there exists $\z^k \neq \0$ such that
$(\z^k)\tr (\mbf{H} + k \mbf{D}\tr \mbf{D}) \z^k \le 0$ and
$\mbf{C} \z^k = \0$.
Assume without loss of generality that $\norm{\z^k} = 1$ (we can scale $\z^k$ as necessary).
Since $\seq[k]{\z^k}$ is in a compact set, we have a subsequence converging to some point $\bar{\z}$ with $\norm{\bar{\z}} = 1$ and $\mbf{C}\bar{\z} = \0$.
Taking the limit superior of 
$(\z^k)\tr (\mbf{H} + k \mbf{D}\tr \mbf{D}) \z^k \le 0$
over this subsequence, we get
\begin{equation}
\label{ineq:PD:toContradict}
\bar{\z} \tr \mbf{H} \bar{\z} + \limsup_k  k (\z^k)\tr \mbf{D}\tr \mbf{D} \z^k \le 0.
\end{equation}
Since $(\z^k)\tr \mbf{D}\tr \mbf{D} \z^k \ge 0$ for all $k$, we must have (the subsequence)
$\{(\z^k)\tr  \mbf{D}\tr \mbf{D} \z^k\}$ 
converging to zero, or else the limsup would be infinite.
Thus $(\mbf{D}\bar{\z})\tr \mbf{D} \bar{\z} = 0$, which implies $\mbf{D}\bar{\z} = \0$.
But by hypothesis this means $\bar{\z}\tr \mbf{H}\bar{\z} > 0$.
Combined with the fact that $\limsup_k k (\z^k)\tr \mbf{D}\tr \mbf{D} \z^k$ must be nonnegative (since each term is nonnegative),
this contradicts Inequality~\eqref{ineq:PD:toContradict}.
\end{proof}

\begin{lemma}
\label{lem:local_reg1}
Let Assumption~\ref{assm:OverallRegularity} hold.
Then for all sufficiently large $\rho$, 
Assumption~\ref{assm:SPRegularity} holds.
\end{lemma}
\begin{proof}
If $(\x^*,\y^*,\bmu^*,\blam^*)$ is a KKT point of the overall problem~\eqref{problem}, then we have
\begin{align*}
\grad f(\x^*) + 
	\grad \mbf{c}(\x^*) \bmu^* + 
	\mbf{A} \tr \blam^* &= \0, \\
	\c(\x^*) &= \0.
\end{align*}
Since we have 
$\mbf{A}\x^* + \mbf{B}\y^* = \b$, 
we can add
$\mbf{A}\tr (\rho (\mbf{A}\x^* + \mbf{B}\y^* - \b))$
to the first equation to get
\begin{align*}
\grad f(\x^*) + 
	\grad \mbf{c}(\x^*) \bmu^* + 
	\mbf{A} \tr (\blam^* + \rho (\mbf{A}\x^* + \mbf{B}\y^* - \b)) &= \0, \\
	\c(\x^*) &= \0,
\end{align*}
which we recognize as the KKT conditions of the subproblem when 
$\y^k = \y^*$ and $\blam^k = \blam^*$.
(see Equations~\eqref{spKKT}).

Differentiability of $f$ and $\c$ and the linear independence constraint qualification 
for the subproblems follow directly from the conditions of Assumption~\ref{assm:OverallRegularity}.
It remains to show that the second order sufficient conditions hold.
Let $\mbf{H}^*$ equal the Hessian of the Lagrangian of Problem~\eqref{problem} at the given KKT point;
that is, let
\begin{equation}
\notag
\mbf{H}^* = \begin{bmatrix} \mbf{H}_{xx}(\x^*,\bmu^*) & \0 \\ \0 & \0 \end{bmatrix}
\end{equation}
(where $\mbf{H}_{xx}$ is defined in Assumption~\ref{assm:OverallRegularity}).
By Assumption~\ref{assm:OverallRegularity}, for any $\z$ satisfying
$\z \neq \0$, 
$\mbf{C}\z = \0$,
we must have
$\z \tr \mbf{H}^* \z > 0$
(where, again, $\mbf{C}$ is defined in Assumption~\ref{assm:OverallRegularity}).
Noting the form of $\mbf{C}$, by Lemma~\ref{lem:PD} this means that there exists $\rho^*$ such that for all $\rho > \rho^*$,
\[
\z \tr (\mbf{H}^* + \rho [\mbf{A}\; \mbf{B}] \tr [\mbf{A}\; \mbf{B}]) \z > 0,
\]
for all $\z$ with
$\z \neq \0$ and 
$\left[\grad \c(\x^*)\tr \;\; \0 \right] \z = \0$.
In particular, this means that for any
$\z_x \in \mbb{R}^{n}$ satisfying
$\z_x \neq \0$ and 
$\grad \c(\x^*) \tr \z_x = \0$,
we have $\z = (\z_x, \0)$ satisfies
\[
0 < \z   \tr \left(\mbf{H}^*                 + \rho  [\mbf{A}\; \mbf{B}] \tr [\mbf{A}\; \mbf{B}]\right) \z  
  = \z_x \tr \left(\mbf{H}_{xx}(\x^*,\bmu^*) + \rho   \mbf{A}            \tr  \mbf{A}\right) \z_x.
\]
Finally we note that $\mbf{H}_{xx}(\x^*,\bmu^*) + \rho \mbf{A} \tr \mbf{A}$ is the Hessian of the Lagrangian of the subproblem~\eqref{ADMMSP} evaluated at $(\x^*, \bmu^*)$.
\end{proof}

\section{A result in parametric optimization}
\label{sec:technicalLemma}

Required by the proof of Lemma~\ref{lem:primalProperties}, the following is a technical result, although it relies on standard and straightforward results.
It is a modification of a classic sufficiency result for local optimality in the parametric setting, stating that there is a minimum size neighborhood on which local optimality holds, for all problems in a perturbed family.

\begin{lemma}
\label{lem:MinimizerRadius}
Let 
$h : (\z,\p) \mapsto h(\z,\p)$
be a real-valued mapping 
(on $\mbb{R}^{n_z} \times \mbb{R}^{n_p}$)
such that
$h$ is twice-continuously differentiable with respect to $\z$ on some open set $D_z$, for all $\p$ in some open set $D_p$.
In addition, assume that
$\grad_{zz}^2 h$ 
is continuous on $D_z \times D_p$.
Assume that for all $\p \in D_p$, there exists $\z^*(\p) \in D_z$ such that
$\grad_z h(\z^*(\p),\p) = \0$
and
$\grad_{zz}^2 h(\z^*(\p),\p)$ is positive definite.
Then for any $\bar\p \in D_p$, there exist positive constants $\epsilon$ and $\delta$
such that for all $\p \in N_{\delta}(\bar\p)$,
if $\norm{\z^*(\p) - \z^*(\bar\p)} \le \epsilon$,
then $\z^*(\p)$ minimizes $h(\cdot,\p)$ on the neighborhood $N_{\epsilon}(\z^*(\p))$.
\end{lemma}
\begin{proof}
That $\z^*(\p)$ minimizes $h(\cdot,\p)$, for all $\p$, follows from the standard second-order sufficient conditions for unconstrained minimization;
see, for instance \cite[Prop.~1.1.3]{bertsekas_nlp}.
The challenge is to show that the radius of the neighborhood on which it is a minimizer is constant with respect to $\p$.
Choose $\bar\p \in D_p$.
Since $\grad_{zz}^2 h(\z^*(\bar\p),\bar\p)$ is positive definite and
$\grad_{zz}^2 h$ is continuous, for all $(\z,\p)$ sufficiently close to $(\z^*(\bar\p),\bar\p)$,
$\grad_{zz}^2 h(\z,\p)$ is positive definite.
In particular, we can choose $\epsilon'$, $\delta$ so that
\[
K = \set{ (\z,\p) : \norm{\z - \z^*(\bar\p)} \le \epsilon', \norm{\p - \bar\p} \le \delta } 
	\subset D_z \times D_p
\]
and $\grad_{zz}^2 h(\z,\p)$ is positive definite for all $(\z,\p)\in K$.
Since the eigenvalues of a matrix depend continuously on the elements of a matrix (\cite[Proposition~A.14]{bertsekas_nlp}), we have that the eigenvalues of $\grad_{zz}^2 h(\z,\p)$ 
(and in particular the \emph{minimum} eigenvalue) 
are continuous and positive for all $(\z,\p)\in K$, and since $K$ is compact,
we can choose a constant $\lambda > 0$ which is a lower bound on the minimum eigenvalue for all $(\z,\p)\in K$.

Now choose any 
$\p \in N_{\delta}(\bar\p)$ and assume 
$\norm{\z^*(\p) - \z^*(\bar\p)} \le \sfrac{\epsilon'}{2}$.
Consider a Taylor expansion of $h(\cdot,\p)$ at $\z^*(\p)$:
for any $\s$ such that $\norm{\s} \le \sfrac{\epsilon'}{2}$, there exists $\alpha_{s,p} \in (0,1)$ such that
\[
	h(\z^*(\p)+\s,\p) = 
	h(\z^*(\p),\p) 
	+ \frac{1}{2} \s\tr \grad_{zz}^2 h(\z^*(\p) + \alpha_{s,p}\s,\p) \s
\]
where the linear term may be ignored because $\grad_z h(\z^*(\p),\p) = \0$.
Note that $(\z^*(\p) + \alpha_{s,p}\s,\p) \in K$, no matter what the specific value of $\alpha_{s,p}$ is.
Consequently, we can use the lower bound $\lambda$ on the minimum eigenvalue of the Hessian to see that
\[
	h(\z^*(\p)+\s,\p) - h(\z^*(\p),\p) 
	\ge
	\frac{1}{2} \lambda \norm{\s}^2
\]
(see for instance \cite[Prop.~A.18]{bertsekas_nlp}).
Define $\epsilon \equiv \sfrac{\epsilon'}{2}$.
The right-hand side of the above inequality is nonnegative, showing that 
$\z^*(\p)$ is a minimizer of $h(\cdot,\p)$ on the neighborhood $N_{\epsilon}(\z^*(\p))$.
\end{proof}


\bibliography{%
books,%
convex,%
nlp,%
special}

\end{document}